\documentclass{amsart}
\usepackage{amsmath}
\usepackage{amssymb} 
\usepackage{amscd} 
\usepackage[dvipdfmx]{graphicx} 
\usepackage{epsfig}
\usepackage[]{xcolor}
\usepackage{marvosym}
\allowdisplaybreaks 
\usepackage{enumerate}
\usepackage[normalem]{ulem} 

\usepackage{booktabs} 

\usepackage{array}
\newcolumntype{M}[1]{>{\centering\arraybackslash}m{#1}} 

\newcommand{\Z}{\mathbb Z}
\newcommand{\N}{\mathbb N}
\newcommand{\Q}{\mathbb Q}

\newcommand{\bdry}{\partial}
\newcommand{\cS}{\mathcal{S}_n}

\newtheorem{theorem}{Theorem}[section]
\newtheorem{lemma}[theorem]{Lemma}

\newtheorem{proposition}[theorem]{Proposition}
\newtheorem{corollary}[theorem]{Corollary}

\newtheorem{conjecture}[theorem]{Conjecture}
\newtheorem{question}[theorem]{Question}

\newtheorem*{thm_maximal_set}{Theorem~\ref{maximal_set}}

\newtheorem*{prop_max-degree}{Proposition~\ref{max-degree}}

\theoremstyle{definition}
\newtheorem{definition}[theorem]{Definition}
\newtheorem{example}[theorem]{Example}

\newtheorem{remark}[theorem]{Remark}

\numberwithin{equation}{section}
\numberwithin{figure}{section}
\numberwithin{table}{section}

\definecolor{dartmouthgreen}{rgb}{0.05, 0.5, 0.06}

\begin{document}
\baselineskip 14pt

\title{The strong slope conjecture for cablings and connected sums}

\author[K.L. Baker]{Kenneth L. Baker}
\address{Department of Mathematics, University of Miami, 
Coral Gables, FL 33146, USA}
\email{k.baker@math.miami.edu}

\author[K. Motegi]{Kimihiko Motegi}
\address{Department of Mathematics, Nihon University, 
3-25-40 Sakurajosui, Setagaya-ku, 
Tokyo 156--8550, Japan}
\email{motegi.kimihiko@nihon-u.ac.jp}

\author[T. Takata]{Toshie Takata}
\address{Graduate School of Mathematics, Kyushu University, 
744 Motooka, Nishi-ku, Fukuoka 819--0395, Japan}
\email{ttakata@math.kyushu-u.ac.jp}

\dedicatory{}

\begin{abstract}
We show that, under some technical conditions, 
the Strong Slope Conjecture proposed by Kalfagianni and Tran is closed under connect sums and cabling. 
As an application, we establish the Strong Slope Conjecture for graph knots. 
\end{abstract}
\maketitle

\renewcommand{\thefootnote}{}
\footnotetext{2010 \textit{Mathematics Subject Classification.}
Primary 57M25, 57M27
\footnotetext{ \textit{Key words and phrases.}
colored Jones polynomial, Jones slope, boundary slope, slope conjecture, strong slope conjecture, 
cabling, connected sum, graph knot}
}

\section{Introduction}
\label{introduction}

Let $K$ be a knot in the $3$--sphere $S^3$.   
The Slope Conjecture due to Garoufalidis \cite{Garoufalidis} and the Strong Slope Conjecture of Kalfagianni and Tran \cite{KT} propose relationships between the degrees of the colored Jones function of $K$ and the essential surfaces in the exterior of $K$.

The \textit{colored Jones function} of $K$ is a sequence of 
polynomials $J_{K, n}(q)$ with $\frac {J_{K, n}(q)}{J_{\bigcirc, n}(q)} \in \mathbb{Z}[q^{\pm 1}]$ for $n \in \mathbb{N}$, 
where $J_{\bigcirc, n}(q)=\frac {q^{n/2}-q^{-n/2}}{q^{1/2}-q^{-1/2}}$ for the unknot $\bigcirc$ and 
$\frac {J_{K, 2}(q)}{J_{\bigcirc, 2}(q)}$ is the ordinary normalized Jones polynomial of $K$. 
Since the colored Jones function is $q$--holonomic \cite[Theorem~1]{GL}, 
the degrees of its terms are given by \textit{quadratic quasi-polynomials} for suitably large $n$ \cite[Theorem 1.1 \& Remark 1.1]{Gqqp}.   
For the maximum  degree $d_+[J_{K,n}(q)]$, 
we set its quadratic quasi-polynomial to be 
\[ \delta_K(n) = a(n) n^2 + b(n) n+ c(n) \]
for rational valued periodic functions $a(n), b(n), c(n)$  with integral period.  
The {\em period} of a quasi-polynomial is the least common period of its coefficient functions.
 Now define the set of {\em Jones slopes} of $K$:
\[  js(K) = \{ 4a(n) \ |\ n \in \mathbb{N} \} .\]
Allowing surfaces to be disconnected, 
we say a properly embedded surface in a $3$--manifold is {\em essential} if each component is orientable, incompressible, boundary-incompressible, and not isotopic into the boundary. 
A number $p/q \in \Q \cup \{\infty\}$ is a {\em boundary slope} of a knot $K$ if there exists an essential surface in the knot exterior $E(K)=S^3-\mathrm{int}N(K)$ with a boundary component representing $p[\mu]+q[\lambda] \in H_1(\bdry E(K))$ with respect to the standard meridian-longitude pair $(\mu, \lambda)$.  
Now define the set of  boundary slopes of $K$:
\[ bs(K) = \{ r \in \mathbb{Q}\cup \{\infty\}\ |\ r\ \mbox{ is a boundary slope of }\ K\}. \]
Since a Seifert surface of minimal genus is an essential surface, 
$0 \in bs(K)$ for any knot. 
Let us also remark that $bs(K)$ is always a finite set \cite[Corollary]{Hat1}. 

Garoufalidis conjectures that Jones slopes are boundary slopes.
\begin{conjecture}[\textbf{The Slope Conjecture} \cite{Garoufalidis}]
\label{slope conjecture}
For any knot $K$ in $S^3$,  
every Jones slope is a boundary slope.  That is $js(K) \subset bs(K)$. 
\end{conjecture}

Garoufalidis' Slope Conjecture concerns only the quadratic terms of $\delta_K(n)$.
Recently Kalfagianni and Tran have proposed the Strong Slope Conjecture which subsumes the Slope Conjecture and asserts that the topology of the surfaces whose boundary slopes are Jones slopes may be predicted by the linear terms of $\delta_K(n)$.

Let $K$ be a knot in $S^3$ with $\delta_K(n) = a(n) n^2 + b(n) n+ c(n)$.
We say that a Jones slope $p/q \in js(K)$ (for $p,q$ coprime and $q>0$) {\em satisfies $SS(n)$} for an integer $n \in \N$ if there is an essential surface $F_n$ in the exterior of $K$ such that 
	\begin{itemize}
		\item $4a(n)=p/q$ is the boundary slope of $F_n$ and
		\item $\displaystyle 2b(n) = \frac{\chi(F_n)}{|\bdry F_n| q}$.
	\end{itemize}
Because $\delta_K(n)$ is a quasi-polynomial, 
we may regard the integer $n$ as its equivalence class modulo the period of $\delta_K(n)$ 
(or more precisely the least common multiple of the periods of $a(n)$, $b(n)$ and $c(n)$).

\begin{conjecture}[\textbf{The Strong Slope Conjecture} \cite{KT,K-AMSHartford,K}]
	\label{SSC}
	For any knot $K$ in $S^3$,  
every Jones slope satisfies $SS(n)$ for some $n \in \N$.
\end{conjecture}

The Strong Slope Conjecture was verified for the following knots:
\begin{itemize}
\item torus knots \cite{Garoufalidis} and their cables \cite{KT},
\item $B$--adequate knots (hence adequate knots and alternating knots) \cite{FKP} and their cables \cite{KT},
\item certain families of $3$-tangle pretzel knots \cite{LV},
\item certain families of Montesinos knots \cite{GLV,LYL},
\item 8, 9-crossing non-alternating knots \cite{KT}, ($9_{47}, 9_{48}$ \cite{Ho}, see also  \cite{LV}), and
\item twisted generalized Whitehead doubles of torus knots, $B$--adequate knots \cite{BMT}.
\end{itemize}

\medskip

For the main purpose of this article, 
we only need to address the Strong Slope Conjecture in the cases of knots $K$ for which  both coefficients $a(n)$ and $b(n)$ of $\delta_K(n)$ are constant functions.  
In this situation Conjecture~\ref{SSC} says that every Jones slope satisfies $SS(1)$. 
Nonetheless, our techniques allow for considerations of knots in which either $\delta_K(n)$ has period at most $2$ or $a(n)$ is constant. 

\medskip

A \textit{graph knot} is a knot obtained from the unknot by a finite sequence of operations of cabling and connected sum. 
These are the knots in $S^3$ whose exterior is a graph manifold, 
a manifold that decomposes along embedded tori into Seifert fibered pieces; cf.\ \cite[Corollary~4.2]{Gor}.

In \cite[Section 2.2]{KL} it is implied that \cite{MT} settles the Strong Slope Conjecture for graph knots and, more generally, for connected sums of knots that satisfy the Strong Slope Conjecture. However, the Strong Slope Conjecture was never discussed and Euler characteristics of essential surfaces and linear coefficients of quadratic quasi-polynomials were not considered in \cite{MT}.   

To rectify this,   we first address the Strong Slope Conjecture for connected sums in Theorem~\ref{strong_slope_conjecture_sum}. Then, after clarifying the behavior of the maximum degree of the colored Jones polynomial for cables of certain knots in Proposition~\ref{max-degree}, we record an explicit proof of the Strong Slope Conjecture for graph knots with Corollary~\ref{strong_slope_conjecture_graph}. 
In particular, applying Theorem~\ref{strong_slope_conjecture_sum} and Proposition~\ref{max-degree}
together with the technical conditions of
Condition $\delta$ and the Sign Condition, 
we prove 

\begin{theorem}
\label{maximal_set}
Let $\mathcal{K}$ be the maximal set of knots in $S^3$ of which each is either the trivial knot or satisfies Condition $\delta$, the Sign Condition, and the Strong Slope Conjecture.
The set $\mathcal{K}$ is closed under connected sum and cabling.
\end{theorem}

As we will observe in Section~\ref{graph_knots}, 
torus knots and $B$--adequate knots belong to $\mathcal{K}$.
Hence Theorem~\ref{maximal_set} immediately implies the following. 

\begin{corollary}
\label{strong_slope_conjecture_graph}
	Every graph knot satisfies the Strong Slope Conjecture.
\end{corollary}

Also, let us note that while there are many non-trivial knots which do not satisfy Condition $\delta$, 
 there are even some knots which do not satisfy the Sign Condition; see Section~\ref{no_Sign_Condition}.

\begin{remark}
\label{d+=delta}
Kalfagianni and Tran give $\delta_{K_{p, q}}(n)$ for a $(p, q)$--cable of a knot $K$ 
when (1) $\delta_K(n)$ has period at most $2$ and $b(n) \le 0$ \cite[Proposition~3.2]{KT}, 
or (2) $a(n)$ is constant, $b(n) \le 0$ \cite[Proposition~4.4]{KT}. 
In their results they do not assume the Sign Condition, but implicitly assume that 
$d_+[J_{K, n}(q)] = \delta_K(n)$ for all $n > 0$ (in their proof). 
Since after taking cables $d_+[J_{K_{p, q}, n}(q)] = \delta_{K_{p, q}}(n)$ holds only for sufficiently large $n$, 
we cannot apply their results to obtain $\delta_{K'}(n)$ for further cables $K'$. 
On the other hand, to prove Theorem~\ref{maximal_set} we need to take iterated cables. 
This leads us to show 
Propositions~\ref{max-degree} and {max-degree-monoslope} in this article where 
we weaken the assumption ``$d_+[J_{K, n}(q)] = \delta_K(n)$ for all $n > 0$'' to 
``$d_+[J_{K, n}(q)] = \delta_K(n)$ for sufficiently large $n > 0$'' by introducing the extra condition, 
the Sign Condition, 
which allows us to deal with the lack of information of $d_+[J_{K, n}(q)]$ for not sufficiently large integers $n > 0$. 
\end{remark}

\bigskip

We close the introduction by clarifying our usage of notation. 
Associated to a knot $K$ is a quadratic quasi-polynomial $\delta_K(n)$ such that there is an integer $N_K$ for which $d_+[J_{K, n}(q)] = \delta_K(n)$  for integers $n \geq N_K$.  
Note that, based on $q$--holonomicity alone, $d_+[J_{K, n}(q)]$ is not necessarily a quadratic quasi-polynomial itself as made explicit by Proposition~\ref{d_+ not delta}.

\bigskip

\noindent
\textbf{Acknowledgments.}
We would like to thank Effie Kalfagianni \cite{K} for discussing the proof of \cite[Proposition~3.2]{KT} and the statement of the Strong Slope Conjecture,
Stavros Garoufalidis for clarifications about the $q$--holonomicity of $d_+[J_{K, n}(q)]$, 
and Christine Lee for sharing her knowledge of counterexamples to \cite[Conjecture 1.4]{LV}.  
We also would like to thank Masaaki Suzuki for suggesting the use of the Mathematica package \verb!KnotTheory`! and its program \verb!ColouredJones! \cite{BarKT,Mat}, which lead us to find examples of knots not satisfying the Sign Condition.

\smallskip 
\noindent
KLB was partially supported by a grant from the Simons Foundation (\#523883 to Kenneth L.\ Baker).
KM was partially supported by JSPS KAKENHI Grant Number 19K03502 and Joint Research Grant of Institute of Natural Sciences at Nihon University for 2019.
TT was partially supported by JSPS KAKENHI Grant Number 17K05256. 

\medskip

\section{The Strong Slope Conjecture and connected sums of knots}
\label{connected_sum}

\begin{theorem}
\label{strong_slope_conjecture_sum}
Let $K_1$ and $K_2$ be knots each of which has a single Jones slope.  
Assume each of these Jones slopes satisfy $SS(n_0)$ for the integer $n_0$.
Then a connected sum $K_1 \sharp K_2$ also has a single Jones slope and this Jones slope satisfies $SS(n_0)$.
\end{theorem}

\begin{proof}

Write 
\[\delta_{K_i}(n) = a_i(n) n^2 + b_i(n) n + c_i(n)\]
for each $i=1,2$.
Since $K_i$ has a single Jones slope $p_i/q_i$ ($q_i > 0$), 
$a_i(n)=a_i$ is constant.
By hypothesis, $p_i/q_i$ satisfies $SS(n_0)$ for some $n_0 \in \N$.
In particular, there is an essential surface $S_i$ properly embedded in $E(K_i)$ that has boundary slope $4a_i = p_i/q_i$ 
and satisfies $\displaystyle \frac{\chi(S_i)}{|\partial S_i| q_i} = 2 b_i(n_0)$. 

First we recall from the proof of \cite[Lemma 2.1]{MT} that
\begin{align*}
\delta_{K_1 \sharp K_2}(n) 
&= \delta_{K_1}(n) + \delta_{K_2}(n) - \frac{1}{2}n + \frac{1}{2} \\
&= (a_1(n) + a_2(n))n^2 + (b_1(n) + b_2(n) -\frac{1}{2}) n + c_1(n) + c_2(n) + \frac{1}{2} \\
&=  (a_1 + a_2)n^2 + (b_1(n) + b_2(n) -\frac{1}{2}) n + c_1(n) + c_2(n) + \frac{1}{2}
\end{align*}
though we have the extra terms $- \frac{1}{2}n + \frac{1}{2}$ due to our use of the unnormalized 
colored Jones function which is addressed in \cite{KT}.    
Thus the quadratic quasi-polynomial of this connected sum also has a constant quadratic term and hence a single Jones slope. 
Note that its linear term is $b(n)=b_1(n) + b_2(n) -\frac{1}{2}$ and has period that divides the least common multiple of the periods of $b_1$ and $b_2$.  
(Actually, let $p$ be the period of $b(n)$ and $p'$ be the least common multiple of the periods $p_i$ of $b_i(n)$.
Then $b(n+p') = b_1(n+p') + b_2(n+p') = b_1(n) + b_2(n) = b(n)$, and thus $p \le p'$. 
Writing $p' = pk+ r\ (0 \le r < p)$, 
we have $b(n) = b(n+p') = b(n + (pk+ r)) = b((n+ r) + pk) = b(n+r)$ for any $n$.  
This shows that $b(n)$ has period $r < p$, and hence $r=0$ and $p$ divides $p'$.)

We next show that $K_1 \sharp K_2$ satisfies $SS(n_0)$. 

Recall that $E(K_1 \sharp K_2)$ is decomposed into $E(K_1)$ and $E(K_2)$ along an essential annulus $A$ whose core is meridian of $K_1$ and $K_2$; 
see \cite[Figure 2.1]{MT}. 
Gluing $m_1$ copies of $S_1$ and $m_2$ copies of $S_2$ along $A$, 
we obtain a (possibly non-orientable) surface $S = m_1 S_1 \cup m_2 S_2$ in $E(K_1 \sharp K_2)$. 
The gluing condition requires that $m_1 |\partial S_1 | q_1 = m_2 |\partial S_2| q_2$.  
As shown in \cite[Lemma 2.2]{MT} the surface $S$ is essential, 
and the boundary slope of $S$ is $p_1/q_1 + p_2/q_2$ which equals $4(a_1 + a_2)$. 
Let us write $p/q = p_1/q_1 + p_2/q_2$ for coprime integers $p, q$ with $q > 0$.

We note that, by construction, 
$m_1 |\partial S_1 | q_1 = m_2 |\partial S_2| q_2$   
equals the number of arcs of $S \cap A$. 
Hence it must also coincide with $|\partial S|q$.   
Thus we have
\begin{itemize}
\item
$m_1 |\partial S_1 | q_1 = m_2 |\partial S_2| q_2 = |S \cap A| = |\partial S| q$, and
\item
$\chi(S) = \chi(m_1 S_1 \cup m_2 S_2) = m_1 \chi(S_1) + m_2 \chi(S_2) - |S \cap A|$.
\end{itemize}
Then it follows that 
\begin{align*}
\frac{\chi(S)}{|\partial S| q} 
&= \frac{m_1 \chi(S_1) + m_2 \chi(S_2) - |S \cap A|}{|\partial S| q} \\
&= \frac{m_1\chi(S_1)}{m_1|\partial S_1|q_1} + \frac{m_2 \chi(S_2)}{m_2 |\partial S_2|q_2} -\frac{|\partial S| q}{|\partial S| q} \\
&= \frac{\chi(S_1)}{|\partial S_1|q_1} + \frac{\chi(S_2)}{|\partial S_2|q_2} -1 \\
&= 2b_1(n_0) + 2b_2(n_0) -1 \\
&= 2(b_1(n_0) + b_2(n_0) -1/2)\\
&= 2b(n_0).
\end{align*}

If $S$ is non-orientable, 
we need to replace $S$ by the frontier $\widetilde{S}$ of a tubular neighborhood $N(S)$ of $S$ in $E(K_1 \sharp K_2)$, 
which is a twisted $I$--bundle over $S$. 
However, as described in \cite[Lemma~5.1]{BMT},  
$\widetilde{S}$ also has boundary slope $p/q$ and 
\[
\frac{\chi(\widetilde{S})}{|\partial \widetilde{S}| q}  
= \frac{\chi(S)}{|\partial S| q} 
= 2(b_1(n_0) + b_2(n_0) -1/2) = 2b(n_0)
\]
as desired.
\end{proof}

\section{The Strong Slope Conjecture and cablings --- a revision of Kalfagianni-Tran's results} 
\label{cabling}

For coprime integers $p,q$ with $q \neq 0$, 
let $K_{p,q}$ be the $(p,q)$-cable knot of a knot $K$. 
That is, $K_{p,q}$ is a curve in the boundary of a solid torus neighborhood of $K$ that, 
with respect to the standard meridian and longitude of $K$, 
winds $p$ times meridionally and $q$ times longitudinally. 
Since $K_{p,\pm1}=K$, we assume $|q|>1$.  
Because the colored Jones function is unchanged by reversing the orientation of a knot, 
we restrict to considering unoriented knots.
Thus we may further assume $q > 1$; see the second paragraph of \cite[Proof of Proposition~3.2]{KT}.

As we mentioned in Remark~\ref{d+=delta}, to prove Theorem~\ref{maximal_set} and Corollary~\ref{strong_slope_conjecture_graph} we need to take iterated cables, 
and thus we need to rectify \cite[Proposition~3.2]{KT} so that we can apply it repeatedly. 
We present our Proposition~\ref{max-degree} as a replacement for \cite[Proposition~3.2]{KT}. 
Our proposition requires the extra technical assumption of the Sign Condition given in Definition~\ref{sign}.  
Our proof of Proposition~\ref{max-degree} below follows the spirit of Kalfagianni and Tran's approach to \cite[Proposition~3.2]{KT}.

\subsection{The Sign Condition and a cabling formula}\label{sec:signcabling}

\begin{definition}[\textbf{The Sign Condition}]
\label{sign}
Let $\varepsilon_n(K)$ be the sign of the coefficient of the term of the maximum degree  of $J_{K, n}(q)$. 
A knot $K$ satisfies the \textit{Sign Condition} if 
$\varepsilon_m(K) = \varepsilon_n(K)$ for $m \equiv n\ \mathrm{mod}\ 2$.
\end{definition}

In Propositions~\ref{ab-twisted-ab-torus} and \ref{prop-BadequateSignCond} we show that 
torus knots and $B$--adequate knots satisfy the Sign Condition.   
In Section~\ref{no_Sign_Condition} we exhibit some knots that fail the Sign Condition. 

\begin{proposition}
\label{max-degree}
Let $K$ be a knot such that 
$\delta_K(n)=a(n)n^2+b(n)n+c(n)$ has period $\le 2$ with $b(n)\le 0$.
Suppose $\frac{p}{q} \neq 4 a(n)$ if $b(n)=0$, 
and $K$ satisfies the Sign Condition.  
Then 
$
\delta_{K_{p,q}}(n) =A(n)n^2+B(n)n+C(n)
$
has period $\le 2$ with
\[
 \{A(n) \} \subset \{q^2 a(q(n-1)+1)\} \cup \{\frac {pq} 4\} \quad and \quad 
 B(n)\le 0.
\]
Explicitly, we have 
\begin{equation*}
  \delta_{K_{p,q}}(n)=
   \left\{ \begin{array}{ll} 
           q^2 a(i) n^2 +\left(q b(i)+\frac {(q-1)(p-4qa(i))} 2\right) n & \\
          \quad\quad +\left(a(i)(q-1)^2-(b(i)+\frac p 2)(q-1)+c(i)\right) & \mbox{for } \frac p q< 4a(i), \\
          \frac {pq(n^2-1)} 4 + C_{{\sigma}}(K_{p,q})  &  \mbox{for }\frac p q  \geq  4a(i),
          \end{array} \right. 
\end{equation*}
where $i \equiv_{(2)} q(n-1)+1$, ${\sigma} \equiv_{(2)} n$, 
and $C_{{\sigma}}(K_{p,q})$
is a number that only depends on the knot $K$, 
the numbers $p$ and $q$, and the parity $\sigma$ of $n$. 
Furthermore, 
$K_{p, q}$ also satisfies the Sign Condition. 
\end{proposition}

\begin{proof}
It will be convenient to extend the colored Jones function to negative integers by the convention that 
$J_{K, -m}(v) = - J_{K, m}(v)$ for integers $m>0$ 
(In the following we use the variable $v$ instead of $q$ to distinguish from the cabling parameter.)  
Note that, with this convention, 
$d_+[J_{K, -m}(v)] = d_+[J_{K, m}(v)]$ for all integers $m\neq0$.  
For notational concision, 
let us also write the periodic coefficients of $\delta_K(m)$ as $a_m = a(m)$, $b_m = b(m)$, and $c_m=c(m)$ for integers $m$ considered mod $2$.  
Furthermore, recall that since the knot $K_{p,q}$ is a non-trivial cable of $K$, 
and our knots are unoriented, 
we may assume $q > 1$; see the first paragraph of Section~\ref{cabling}.

A formula for the colored Jones function of a cable of a component of a link is given in \cite{V1,V2}.  
It is presented for the cable of a knot and adapted to our current notations and normalizations in \cite[Equation (3.2)]{KT} which we now recall.  
To do so we must introduce the following sets. 
For each integer $n>0$, 
 let $\cS$ be the finite set of all numbers $k$ such that
\[
|k|\le \frac {n-1} 2 \quad \mathrm{and} \quad 
k\in \left\{ \begin{array}{ll} 
            \Z & \mathrm{if} \, n \,\mathrm{is \, odd},\\
            \Z+\frac 1 2 & \mathrm{if} \, n \,\mathrm{is \, even}.
            \end{array}
\right.
\]
That is,
\[ \cS = \left\{-\frac{n-1}{2},\ -\frac{n-1}{2}+1,\ -\frac{n-1}{2}+2,\ \dots,\ \frac{n-1}{2} -1,\ \frac{n-1}{2}\right\}.\]
Then, from \cite{V1,V2} and following \cite[Equation (3.2)]{KT}, for $n>0$ we have 
\begin{equation}
\label{cabling.formula}
 J_{K_{p,q}, n}(v)=v^{pq(n^2-1)/4}\sum_{k\in \cS} v^{-pk(qk+1)}J_{K, 2qk+1}(v), 
\end{equation}
where we use the convention introduced above that 
$J_{K, -m}(v) = -J_{K, m}(v)$ for integers $m>0$. 

Since we wish to determine $\delta_{K_{p,q}}(n)$, 
we must determine $d_+[J_{K_{p,q}, n}(v)]$ for  $n \gg 0$.
Based on Formula~(\ref{cabling.formula}), 
\begin{equation}
\label{cabling.max}
d_+[J_{K_{p,q}, n}(v)] = pq(n^2-1)/4 + \max_{k \in {\cS}} \{-pk(qk+1) + d_+[J_{K, |2qk+1|}(v)] \}
\end{equation}
assuming this maximum is uniquely realized.  
If this maximum is not uniquely realized, 
then there may be a cancellation between terms of maximal degrees (corresponding to the highest horizontal dotted line in Figure~\ref{fig:cancellation_fig}). 
This cancellation may cause infinitely many further cancellations in the sum of Formula~(\ref{cabling.formula}) as illustrated in Figure~\ref{fig:cancellation_fig}.   
Observe that, for integers $n>0$ of a given parity, 
the parity of $2qk+1$ for $k \in \cS$ is constant. 
More precisely, if $n$ is odd, then $2qk + 1$ is odd, 
and if $n$ is even, then $2qk+1$ is odd or even according to whether $q$ is even or odd, respectively.
In particular, since $\max \cS = \frac{n-1}{2}$, we have $2qk+1 \equiv_{(2)} q(n-1)+1$.
Hence, the Sign Condition for $K$ ensures that no cancellations occur among terms of maximal degree.  
Thus equation~\ref{cabling.max} holds.

\begin{figure}[h]
	\centering
	\includegraphics[width=0.8\textwidth]{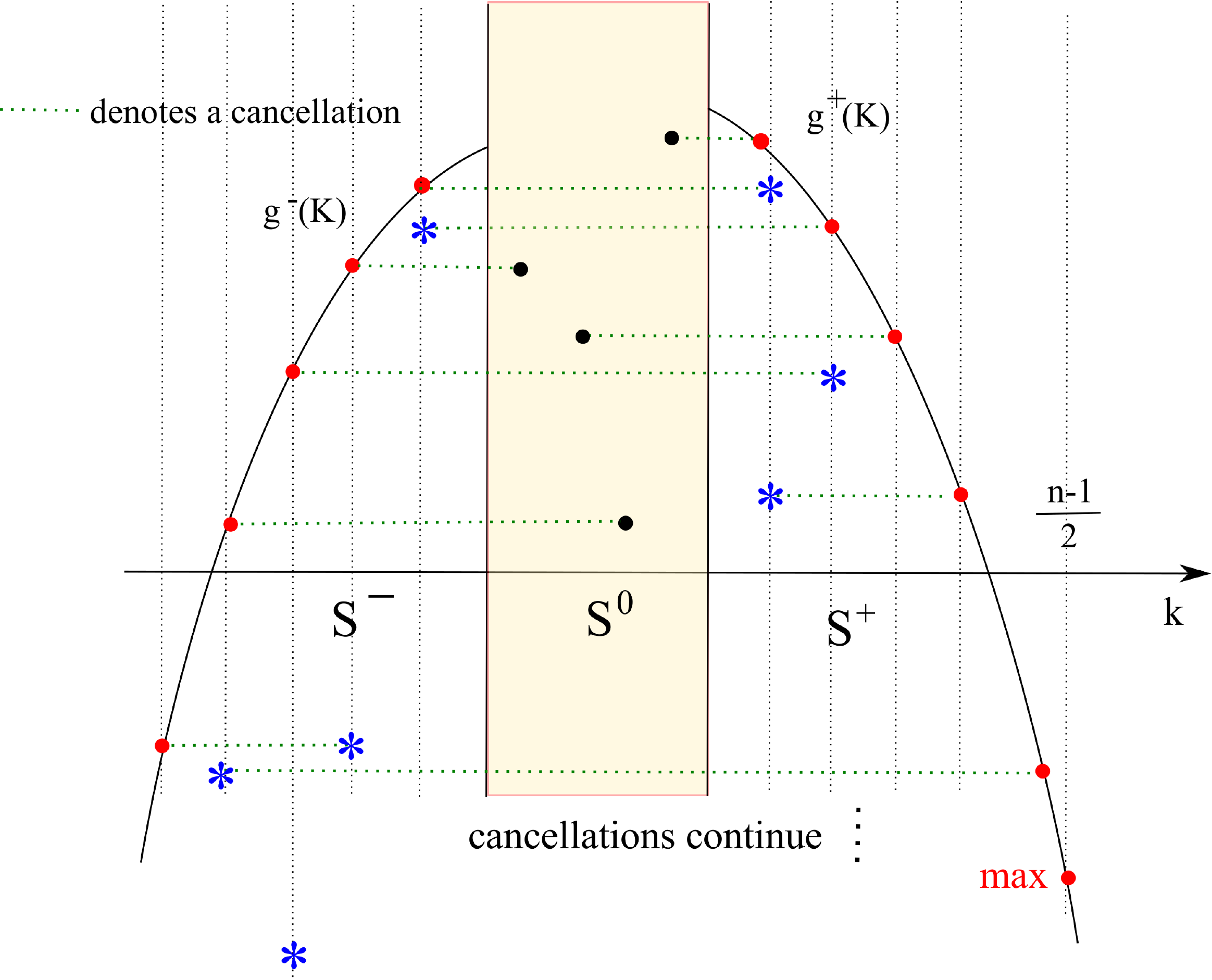}
	\caption{$*$ denotes $m(k)$ for a term $a_{k, m}v^{m(k)}$ of $v^{-pk(qk+1)} J_{K, 2qk+1}(v)$ 
	with $m(k) < f(k)$.}
	\label{fig:cancellation_fig}
\end{figure}

\medskip
First define 
\[
 f(k) = -pk(qk+1) + d_+[J_{K, |2qk+1|}(v)] 
\]
for $k\in \cS$.   
Set $N_K \geq 0$ to be the first integer such that $d_+[J_{K, |m|}(v)]=\delta_K(|m|)$ for all integers $m$ with $|m| \ge 2q N_K +1$. 
Noting that $|2q(-N_K - \frac{1}{2}) + 1| \geq 2qN_K+1 > |2q(-N_K) + 1|$, 
partition $\cS$ into the three subsets 
\[\cS^- = \cS \cap (-\infty, -N_K-\frac{1}{2}], \quad \cS^0 = \cS \cap (-N_K-\frac{1}{2}, N_K), \quad \mbox{ and } \cS^+ = \cS \cap [N_K, \infty). \]
Note that when $n=1$, $\cS = \{ 0 \}$ and $\cS^- = \emptyset$.

Then, considering the quadratic quasi-polynomials for integers and half-integers $k$
\begin{align*}
g^+(k) &= -pk(qk+1) + \delta_K(2qk+1) \\
		&= (-pq + 4 q^2 a_{m}) k^2 + (-p+4q a_{m} + 2q b_{m}) k +(a_{m}+ b_{m}+c_{m})\\
		& \quad \mbox{for } k \geq 0 \mbox{ and } m \equiv_{(2)} 2qk+1
\end{align*}

and
\begin{align*}
g^-(k)  &= -pk(qk+1) + \delta_K(|2qk+1|)\\
         & = -pk(qk+1) + \delta_K(-2qk-1)\\
	&= (-pq + 4 q^2 a_{m}) k^2 + (-p+4q a_{m} - 2q b_{m}) k +(a_{m}- b_{m}+c_{m})\\
	& \quad \mbox{for } k < 0 \mbox{ and } m \equiv_{(2)} |2qk+1|,
\end{align*}
define the quadratic real polynomials $g_m^{\pm}(x)$ for integers $m\pmod 2$ by
\[g_m^{\pm}(x)=(-p q+4 q^2 a_m)x^2+(-p+4q a_m \pm 2q b_m)x+a_m\pm b_m+c_m.\]
Hence for integers and half-integers $k$, 
we have $g^\pm(k) = g^\pm_{m}(k)$ where $m \equiv_{(2)} |2qk+1| $ and $\pm$ means $+$ if $k \geq 0$ and $-$ if $k <0$.
Thus, on the subsets $\cS^\pm$, we have 
\[
f(k)= g_m^\pm(k)  \quad \mathrm{ if }\, k\in \cS^\pm\ \mathrm{ and }\ m\equiv_{(2)} |2qk+1|.
\]
While we have little information about $f(k)$ for $k \in \cS^0$, 
it belongs to only a finite set of values since 
\[\cS^0 \subset \left[-N_K, N_K - \frac 1 2\right] \cap \frac 1 2 \Z\]
for all $n>0$.

\bigskip

Recall that $f(k)$ is defined on half-integers $k$ if $n$ is even and 
defined on integers $k$ if $n$ is odd. 
So to clarify the role of the parity of $n$, 
for each parity $\sigma \in \{0,1\}$ we define the function $f_{{\sigma}}(k)$ so that 
$f(k) = f_{{\sigma}}(k)$ if ${\sigma} \equiv_{(2)} n$. 
Then $f_0$ is defined on half-integers, while $f_1$ is defined on integers. 
More explicitly, this is the function 
\[
f_{{\sigma}}(k)=\begin{cases}
g_i^+(k) &  \mathrm{if} \, k\in \cS^+,\\
g_i^-(k) &  \mathrm{if} \, k\in \cS^-, 
\end{cases}
\]
where $i \equiv_{(2)} q(n-1)+1$.

Henceforth regard ${\sigma}$ as fixed choice of parity.
Note that the parity $i$ is fixed if we vary $n$, maintaining $n \equiv_{(2)} {\sigma}$.
Using (\ref{cabling.max}) and $f_{{\sigma}}(k)$, 
we now proceed to determine $d_+[J_{K_{p,q}, n}(v)]$ for suitably large $n$ such that $n \equiv_{(2)} {\sigma}$.

\smallskip
\noindent 
{\bf Case 1.} 
Assume $\frac p q<4a_i$.
Then $-pq+4q^2a_i>0$,
and so the functions given by the quadratic polynomials $g_i^+(x)$ and $g_i^-(x)$ are concave up.
Hence, for any sufficiently large integer $n$,  
$g_i^+(k)$ is maximized on $\cS^+$ at $k=\frac {n-1} 2$
and $g_i^-(k)$ is maximized on $\cS^-$ at $k=-\frac {n-1} 2$.
Note that 
\[
g_i^+(\frac {n-1} 2)-g_i^-(-\frac {n-1} 2)=(-p+4qa_i)(n-1)+2b_i>0
\]
for sufficiently large integer $n$. 
Therefore, $f_{{\sigma}}(k)$ is maximized on the set $\cS^+ \cup \cS^-$ 
at $k=\frac {n-1} 2$. 

Since the elements of $\cS^0$ belong to a fixed finite set that is independent of $n$, 
the maximum of $f_{{\sigma}}(k)$ on $\cS^0$ has an upper bound that is independent of $n$.
Thus, for a sufficiently large integer $n$, 
we can be assured that $g_i^+(\frac {n-1} 2)$ exceeds this bound.  
Hence
\[ \max_{k\in \cS} f_{{\sigma}}(k) = f_{{\sigma}}(\frac{n-1} 2) = g_i^+(\frac {n-1} 2). \]
Then, Formula~(\ref{cabling.formula}) implies that for sufficiently large integer $n$
\begin{eqnarray*}
	d_+[J_{K_{p,q}, n}(v)]&=& \frac {pq(n^2-1)} 4+ g_i^+(\frac {n-1} 2) \\
	&=& q^2 a_i n^2+\left(qb_i+\frac {(q-1)(p-4qa_i)} 2\right)n \\
	&& \quad \quad +\left(a_i(q-1)^2-(b_i+\frac p 2)(q-1)+c_i\right). 
\end{eqnarray*}
Since we assumed that $q>1$, we have that 
\[
B(n)=qb_i+\frac {(q-1)(p-4qa_i)} 2<0, 
\]
and the conclusion follows in this case.

\smallskip

\noindent 
{\bf Case 2.} 
Assume $p/q>4a_i$.  
Then $-pq+4q^2a_i<0$, 
and so the function given by the quadratic polynomial $g_i^+(x)$ is concave down 
and attains its maximum at 
\[
x=x_0:=-\left( \frac 1 {2q} + \frac {b_i}{-p+4qa_i}\right). 
\]
Since $b_i\le 0$, we have $x_0<0$. 
This implies that $g_i^+(x)$ is a strictly decreasing function on $[0,\infty )$.   
Similarly, the quadratic polynomial $g_i^-(x)$ is concave down 
 and attains its maximum at 
\[ 
x=x_0':=-\left( \frac 1 {2q} - \frac {b_i}{-p+4qa_i}\right). 
\]
Since $b_i\le 0$, we have $x_0' > -\frac 1 2$. 
This implies that $g_i^-(x)$ is a strictly increasing function on $(-\infty,-\frac 1 2]$. 
Thus $g_i^+(k)$ is maximized on $\cS^+$ at $k^+ =\min \cS^+$ and $g_i^-(k)$ is maximized on $\cS^-$ at $k^- = \max \cS^-$.

Since $| \cS^0 | \le 2N_K$, 
there are at most $2N_K$ values $f(k)$ for $k \in \cS^0$, 
and thus we may take $M_0 = \max \{ f(k)\ |\ k \in \cS^0 \}$. 
Now let us put 
$C_{\sigma}(K_{p, q}) = \max\{ f_{{\sigma}}(k^+), f_{{\sigma}}(k^-), M_0 \}$. 
Formula~(\ref{cabling.max}) implies that 
\[
d_+[J_{K_{p,q}, n}(v)]=\frac {pq(n^2-1)} 4+ C_{{\sigma}}(K_{p,q})
\]
for sufficiently large integer $n$ 
with ${{\sigma}} \equiv_{(2)} n$. 

Note that $B(n)=0$.  
Hence the conclusion follows in this case too.

\smallskip

\noindent 
{\bf Case 3.} Assume $p/q=4a_i$ and $b_i<0$.
Then $-p+4qa_i=0$ so that
\[g_i^{\pm}(x)=\pm (2q b_i)x+a_i\pm b_i+c_i.\]
Since $q > 1$ and $b_i < 0$, $g_i^+(x)$ is strictly decreasing and $g_i^-(x)$ is strictly increasing.  
Thus $g_i^+(k)$ is maximized on $\cS^+$ at $k^+ = \min \cS^+$ and $g_i^-(k)$ is maximized on $\cS^-$ at $k^- =\max \cS^-$.

As in Case 1, 
let $M_0$ be $\max \{ f_{{\sigma}}(k)\ |\ k \in \cS^0 \}$ and put 
$C_{{\sigma}}(K_{p, q}) = \max \{ f_{{\sigma}}(k^+), f_{{\sigma}}(k^-), M_0 \}$. 
Then 
\[
d_+[J_{K_{p,q}, n}(v)]=\frac {pq(n^2-1)} 4+ C_{{\sigma}}(K_{p,q})
\]
for sufficiently large integer $n$ 
with ${{\sigma}} \equiv_{(2)} n$.  

\medskip

Finally we show that $\varepsilon_m(K_{p, q}) = \varepsilon_n(K_{p, q})$ for $m \equiv n\ \mathrm{mod}\ 2$. 
From the formula (\ref{cabling.formula}), $J_{K_{p,q}, n}(v)$ has the following form 
\begin{equation*}
 J_{K_{p,q}, n}(v)=v^{pq(n^2-1)/4}\sum_{k\in \cS} v^{-pk(qk+1)}J_{K, 2qk+1}(v). 
\end{equation*}
Since $2qk+1 \equiv q(n-1)+1 \pmod 2$ and the parity of $2qk+1$ for $k \in \cS(n)$ is constant, 
from the assumption for $\varepsilon_n(K)$, 
cancellations in the proof of Proposition 3.1 do not happen and 
we can see that the colored Jones polynomial of $J_{K_{p,q}}(n)$ has the required 
property 
of $\varepsilon_n(K_{p,q})$. 
\end{proof}

\begin{remark}
In Case 3 of the above proof, if we allow $b_i=0$ then $g_i^\pm(x)  =a_i+c_i$, and so it is constant.   
Thus determining $d_+[J_{K_{p,q}, n}(v)]$  from equation~(\ref{cabling.formula}) requires more knowledge of the coefficients of the leading terms in $J_{K, 2qk+1}(v)$ for $2qk+1 \equiv_{(2)} q(n-1)+1$.  
However it is conjectured that $b_i=0$ only when $K$ is cabled \cite[Conjecture 5.1]{KT} 
(via the Strong Slope Conjecture and the Cabling Conjecture \cite{GAS}).  
In such a case, $4a_i$ is an integer so that $p/q \neq 4a_i$ for $q>1$.  
Hence this remaining situation conjecturally does not happen.
\end{remark}

\subsection{Condition $\delta$, cabling, and the Strong Slope Conjecture}
It is also convenient to collect some common assumptions on $\delta_K(n)$ for a knot $K$. 

\begin{definition}[\textbf{Condition $\delta$}]
\label{condition}
We say that a knot $K$ satisfies \textit{Condition $\delta$} if 
\begin{enumerate}
\item
$\delta_K(n)=an^2 + bn + c(n)$ has period at most $2$, 
\item
$b \le 0$, and 
\item $4a  \in \mathbb{Z}$. 
\end{enumerate}
(Note that the trivial knot does {\em not} satisfy Condition $\delta$ because it has $b=1/2$.)
\end{definition}

A version of the following proposition is essentially given in \cite[Theorem 3.9]{KT}.

\begin{proposition}
\label{cable-max}
Let $K$ be a knot that satisfies Condition $\delta$, the Sign Condition, and the Strong Slope Conjecture.  
Then a non-trivial cable $K_{p,q}$ satisfies Condition $\delta$, the Sign Condition, and the Strong Slope Conjecture.
\end{proposition}

\begin{proof}
Due to Condition $\delta$, $\delta_K(n) = an^2 + bn + c(n)$ has period $\leq 2$ with $4a \in \Z$ and $b\leq 0$.  
Since the cable is non-trivial, we have $q>1$ so that $\frac p q \ne 4a$. 
Then, because $K$ also satisfies the Sign Condition, Proposition~\ref{max-degree} shows that 
$\delta_{K_{p,q}}(n)=An^2+Bn+C(n)$ has period $\le 2$, 

(1) if $\frac p q < 4a$, then $A=q^2a$ and $B=qb+\frac {(q-1)(p-4qa)}{2} $,  

(2) if $\frac p q > 4a$, then $A=\frac {pq} 4$ and $B=0$,  

\noindent 
and so, in both cases, $4A \in \Z$ and $B \le 0$. 

Since $K$ satisfies the Strong Slope Conjecture, there is an essential surface $S_K$ in $E(K)$ with boundary slope $4a$ such that  $\displaystyle{\frac {\chi (S_K)}{|\bdry S_K |}=2b}$. 
In case $(1)$, an essential surface $S$ in $E(K_{p,q})$ realizing the boundary slope $4A=4qa^2$ is obtained from $S_K$ in the proof of \cite[Theorem 2.2]{KT}.  
Furthermore \cite[Corollary 2.8]{KT} shows 
 $|\bdry S|=|\bdry S_K|$ and 
\[
  \chi (S)=q \chi (S_K)+|\bdry S_K|(q-1)(p-4aq),
\]
so that we obtain:
\[
 \frac {\chi (S)}{|\bdry S|}=2bq+(q-1)(p-4aq)=2B.
\]
In case (2), the surface $S$ with boundary slope $pq$ is the cabling annulus, thus 
$\frac {\chi (S)}{|\bdry S|}=0=2B$. 
\end{proof}

\begin{remark}
	\label{period}
	One may note that in Propositions~\ref{max-degree} and \ref{cable-max},
	we need not require that the constant coefficient $c(n)$ has period $\leq 2$ in order to obtain the relevant results about the quadratic and linear coefficients of $\delta(n)$.  
	However, the assumption that $\delta(n)$ has period $\leq 2$ does simplify the presentation and proof of Proposition~\ref{max-degree} from which the other is derived.
	In Appendix~\ref{app:monosloped} we consider the Strong Slope Conjecture for cables of knots of arbitrary period but with single Jones slope and the Sign Condition, updating \cite[Theorem 4.1 and Proposition 4.4]{KT}. 
\end{remark}

\subsection{The degree of the colored Jones polynomial is not always a quadratic quasi-polynomial}
Though it is clarified in the text, the title of \cite[Section 1.2]{Garoufalidis} may have caused a misconception.
Using a specialization of Proposition~\ref{max-degree} for torus knots recorded in Proposition~\ref{cable_torus}, we present Example~\ref{example_d_+ not delta} which concretely demonstrates the existence of cabled knots for which the degree of their colored Jones polynomial are quadratic quasi-polynomials only for suitably large integers $n$.  
In particular, for such a knot $K$, $d_+[J_{K,n}](q) \neq \delta_K(n)$ when $n$ is a positive integer below an explicit cut-off that depends on $K$.  Moreover this cut-off can be arbitrarily large.

\begin{proposition}
\label{d_+ not delta}
There exists a knot $K$ with cable $K'$ such that 
$d_{+}[J_{K, n}(q)] = \delta_K(n)$ for all integers $n>0$, 
but $d_{+}[J_{K', n}(q)] \ne \delta_{K'}(n)$ for integers $n=1,2, \dots, N$ where $N$ is a positive integer. 
Moreover, the knot $K'$ may be chosen so that $N$ is larger than any given number. 
\end{proposition}  

\begin{proof}
Example~\ref{example_d_+ not delta} below provides the concrete example of a $(12q^2-1,2)$--cable of a $(6q-1,q)$--cable of the $(3,2)$--torus knot for which the maximal degree of its colored Jones polynomial is a 
quadratic quasi-polynomial only for integers $n\geq 2q-1$ for integers $q>3$.  
(In this example for $0 < n < 2q-2$, the maximal degree of its colored Jones polynomial is another  
quadratic quasi-polynomial.)
However, for the $(6q-1,q)$--cable of the $(3,2)$--torus knot, the maximal degree of its colored Jones polynomial is a 
quadratic quasi-polynomial for all integers $n>0$. 
\end{proof}

Since our construction uses cabling, and noting that $d_+=\delta$ for torus knots, it is natural to wonder if any hyperbolic knot exhibits this behavior.

\begin{question}\label{questionhypd+delta} \phantom{xxx}
\begin{enumerate}
\item
For every hyperbolic knot $K$, does $d_{+}[J_{K, n}(q)] = \delta_K(n)$ for all integers $n > 0$?  
\item
Even when $d_{+}[J_{K, n}(q)] = \delta_K(n)$  only for $n \ge N_K$, 
is $d_{+}[J_{K, n}(q)]$ another quadratic quasi-polynomial for $n < N_K$ as well?  
\end{enumerate}
\end{question}

In preparation for Example~\ref{example_d_+ not delta}, we observe Proposition~\ref{cable_torus} which specializes Proposition~\ref{max-degree} for the case of cables of torus knots. For its proof, we use the notation (such as $g_{i}^+(k)$ and $g_{i}^-(k)$) from  the proof of Proposition~\ref{max-degree}  and follow its argument without further reference. We will also use this notation in the explanation of Example~\ref{example_d_+ not delta}.

\begin{proposition}
\label{cable_torus}
Let $T$ be the $(a,b)$--torus knot with $a,b>1$ and $T_{p,q}$ be a $(p,q)$-cable of $T$ with $q >1$. 
 Then, for any $n>0$, 
$d_+[J_{T_{p,q}, n}(v)]=\delta_{T_{p,q}}(n)$, 
and explicitly 
$\delta_{T_{p,q}}(n)$ is given by 
\begin{equation*}
  \delta_{T_{p,q}}(n)=
   \left\{ \begin{array}{ll} 
           \frac {q^2 ab} 4  n^2 +\frac {(q-1)(p-qab)} 2 n & \\
          \quad\quad + \left(\frac {ab} 4 (q-1)^2-\frac p 2(q-1)-\frac {ab} 4-(1+(-1)^i) \frac {(a-2)(b-2)} 8\right) 
          & \mbox{for } \frac p q< ab, \\
          \frac {pq(n^2-1)} 4 + C_{\sigma}(T_{p,q})  &  \mbox{for }\frac p q  > ab,
          \end{array} \right. 
\end{equation*}
where $i \equiv_{(2)} q(n-1)+1$, $\sigma \equiv_{(2)} n$, 
and  $C_{\sigma}(T_{p,q})$
is a number that only depends on the knot $T$, 
the numbers $p$ and $q$, and the parity $\sigma$ of $n$. 
\end{proposition}

\begin{proof}
First observe that when  (\ref{cabling.max}) is applied to a $(p,q)$--cable $K_{p,q}$ of a knot $K$ for which $d_+[J_{K, n}(v)] = \delta_K(n)$, we obtain
\begin{equation}
\label{cabling.max2}
d_+[J_{K_{p,q}, n}(v)] = pq(n^2-1)/4 + \max_{k \in {\cS}} \{g_i^{\pm }(k) \}.
\end{equation}

Now assume $T$ is an $(a,b)$--torus knot.
Following \cite[Section 4.8]{Garoufalidis} but with our normalization so that it appears as in \cite[Proof of Theorem 3.9, Case 2]{KT}, 
we have the explicit computation 
\[ d_+[J_{T, n}(v)] = \delta_T(n) = \frac {ab} 4n^2-\frac  {ab} 4-(1+(-1)^{n})\frac{(a-2)(b-2)}{8}\] 
for all integers $n>0$. 
Notably, (\ref{cabling.max2}) applies. 
Then we see that 
\[g_m^+(k) = -pk(qk+ 1) + \delta_T(2qk+ 1) = -pk(qk+ 1) + \delta_T(-2qk- 1) = g_m^-(k),\] 
and thus precisely 
\begin{eqnarray*}
g_m^+(k)=g_m^-(k)&=&q(-p+qab)k^2+(-p+qab)k-(1+(-1)^m) \frac {(a-2)(b-2)}{8}\\
  &=&q(-p+qab)(k+\frac 1 {2q})^2-\frac {-p+qab}{4q}-(1+(-1)^m) \frac {(a-2)(b-2)} 8. 
\end{eqnarray*}
Thus using (\ref{cabling.max2}) we compute that for any $n>0$, 
if $\frac p q<ab$, 
\begin{equation*}
d_+[J_{T_{p,q}, n}(v)] =
\frac{p  q(n^2 - 1)}{4} + g_{i}^+(\frac {n-1}2) 
\end{equation*} 
and if $\frac p q> ab$, 
\begin{equation*}
d_+[J_{T_{p,q}, n}(v)] =
\frac{p  q(n^2 - 1)}{4} + C_{\sigma}(T_{p,q}) 
\end{equation*} 
where $i \equiv_{(2)} q(n-1)+1$, $ C_0(T_{p,q})=g_i^-(-\frac 1 2)$ and $ C_1(T_{p,q})=g_i^+(0)$. 
Therefore we have 
\begin{equation*}
d_+[J_{T_{p,q}, n}(v)] = \delta_{T_{p,q}}(n)
\end{equation*}
for all $n>0$.  
\end{proof}

\begin{example}
\label{example_d_+ not delta}
Let $T$ be  the $(3,2)$--torus knot $(q >3)$, 
$T_{6q-1, q}$ a $(6q-1, q)$--cable of $T$ and 
$T_{6q-1, q;\ 12q^2-1, 2}$ a  $(12q^2-1,2)$--cable of $T_{6q-1, q}$. 
Then for $T$ and $T_{6q-1, q}$, 
we have
$d_+[J_{T, n}(v)]= \delta_T(n)$ and $d_+[J_{T_{6q-1, q}, n}(v)]= \delta_{T_{6q-1, q}}(n)$ 
for all  $n > 0$ (Proposition~\ref{cable_torus}). 
However, 
$d_+[J_{T_{6q-1, q;\ 12q^2-1, 2}, n}(v)]= \delta_{T_{6q-1, q;\ 12q^2-1, 2}}(n)$
only for $n \ge 2q-1 $. 
\end{example} 

\begin{proof}

Noting that $\frac {6q-1} q<3\cdot 2$, from Proposition~\ref{cable_torus}, 
we have 
\begin{equation*}
d_+[J_{T_{6q-1,q}, n}(v)] = \delta_{T_{,6q-1q}}(n)
=\frac {3 q^2} 2n^2+\frac {1-q} 2 n-\frac {3q^2-q +1}{2}
\end{equation*}
for all $n>0$. 

Now put $K = T_{6q-1,q}$ and consider ${K}_{12q^2-1,2}$, 
the $(12q^2-1,2)$--cable of $K = T_{6q-1,q}$.  
Then we have
\begin{eqnarray*}
 && g^+(k) = g_{{K}_{12q^2-1,2}}^+(k) 
 =2k^2-(2q-3)k 
 =2(k-\frac {2q-3} {4})^2 -\frac 1 2 q^2+\frac 3 2 q-\frac 9 8  \\
 && g^-(k) = g_{{K}_{12q^2-1,2}}^-(k) 
 =2k^2+(2q-1)k+q -1
 =2(k+\frac {2q-1} {4})^2-\frac 1 2 q^2+\frac 3 2 q-\frac 9 8.
\end{eqnarray*}
See Figure~\ref{fig:not_all_n}.

\begin{figure}[h]
\includegraphics*[width=0.5\textwidth]{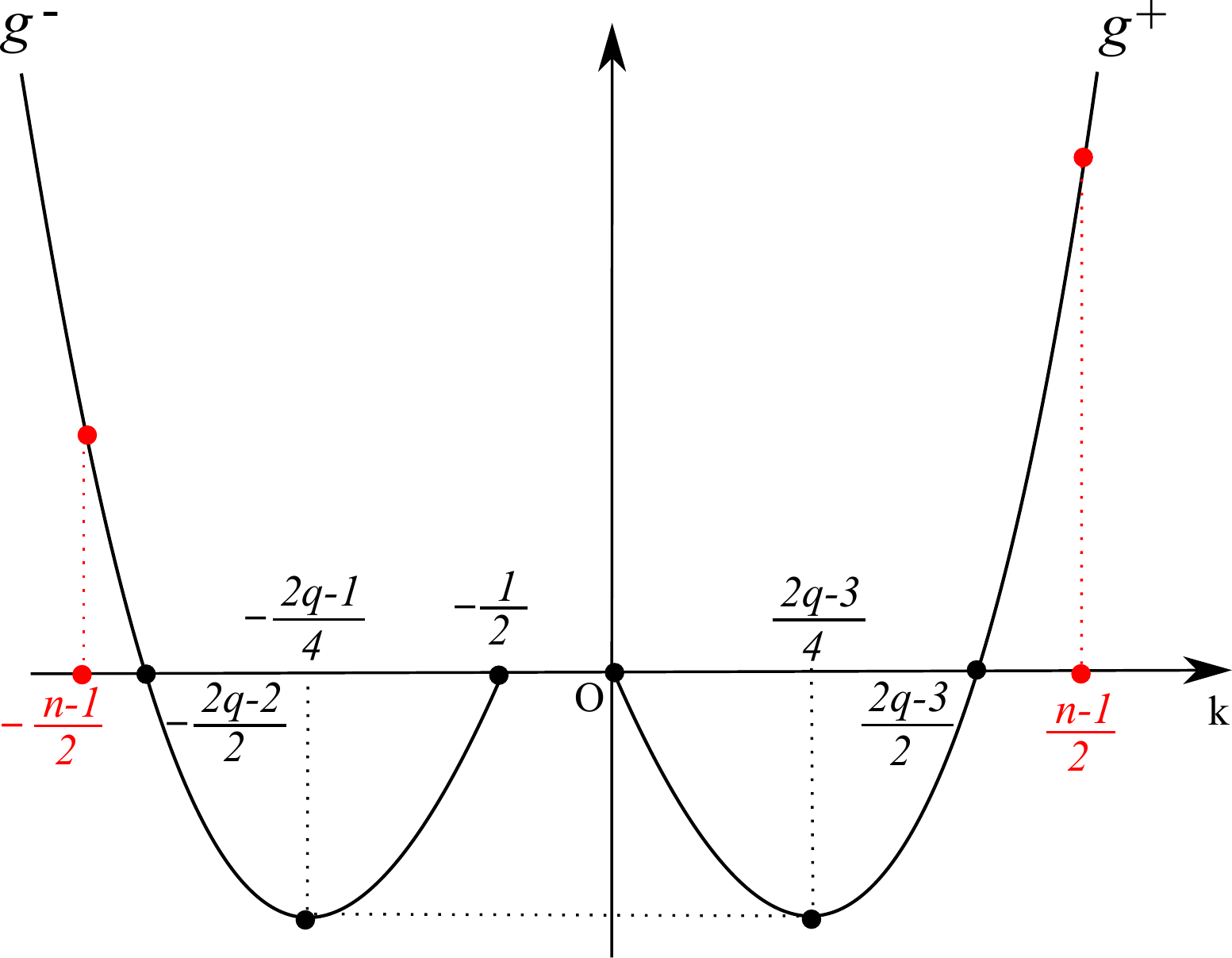}
\caption{graphs of $g^+$ and $g^-$}
\label{fig:not_all_n}
\end{figure}

Then, 
since $d_+[J_{K, n}(v)] = \delta_K(n)$ for all $n>0$, we may use (\ref{cabling.max2}) to compute that
\begin{eqnarray*}
&&d_+[J_{{K}_{12q^2-1,2}, n}(v)] \\
&&=\left\{\begin{array}{ll}
    \frac{(12q^2-1)\cdot 2(n^2 - 1)}{4} + 0
    = \frac{12q^2-1}{2}n^2 -\frac{12q^2-1}{2} & \mathrm{if} \ n\le 2q-2,\\
    \frac{(12q^2-1)\cdot 2(n^2 - 1)}{4} + g_{{K}_{12q^2-1,2}}^+(\frac {n-1} 2) 
        =  6 q^2n^2 + \frac{1-2q}{2}n -6q^2+q-\frac 1 2&  \mathrm{if} \ n\ge 2q-1.                                 
                         \end{array}
                      \right.
\end{eqnarray*}
In particular, 
$d_+[J_{{K}_{12q^2-1,2}, n}(v)] = \delta_{{K}_{12q^2-1,2}}(n)$ only for $n \geq 2q-1$. 
\end{proof}

\section{The Strong Slope Conjecture for graph knots}
\label{graph_knots}

In this section we prove the Strong Slope Conjecture for graph knots (Corollary~\ref{strong_slope_conjecture_graph}) by establishing it for wider class of knots (Theorem~\ref{maximal_set}).
For this we need the the technical conditions of the Sign Condition  (Definition~\ref{sign})  and Condition $\delta$ (Definition~\ref{condition}).

\begin{thm_maximal_set}
Let $\mathcal{K}$ be the maximal set of knots in $S^3$ of which each is either the trivial knot or satisfies Condition $\delta$, the Sign Condition, and the Strong Slope Conjecture.
The set $\mathcal{K}$ is closed under connected sum and cabling.
\end{thm_maximal_set}

\begin{proof}
Theorem~\ref{maximal_set} follows from Lemmas~\ref{condition sum} and \ref{condition cable} below. 
\end{proof} 

\begin{proof}[Proof of Corollary~\ref{strong_slope_conjecture_graph}]
\label{example1}
Let $K$ be a graph knot. 
Then, as noted in the introduction, 
$K$ is obtained from the trivial knot by a finite sequence of operations of cabling 
and connected sum; cf.\ \cite[Corollary~4.2]{Gor}. 
Since the trivial knot is in $\mathcal{K}$ by definition, 
it follows from Theorem~\ref{maximal_set} that the set of nontrivial graph knots is contained in $\mathcal{K}$. 
Thus any graph knot satisfies the Strong Slope Conjecture.

It then follows from Theorem~\ref{maximal_set} that the set of nontrivial graph knots is contained in $\mathcal{K}$. 
Thus any graph knot satisfies the Strong Slope Conjecture. 
\end{proof}

Lemma~\ref{condition sum} and Proposition~\ref{ab-twisted-ab-torus} both make use of a normalization of the colored Jones function.  
For knot $K$ and a nonnegative integer $n$, 
the {\em normalized colored Jones function} of $K$ is the function
\[J'_{K, n}(q):=\frac{J_{K, n+1}(q)}{J_{\bigcirc, n+1}(q)}\] 
so that $J'_{\bigcirc, n}(q)=1$ for the unknot $\bigcirc$ and 
$J'_{K, 1}(q)$ is the ordinary Jones polynomial of a knot $K$. 
In particular, taking $\langle n \rangle$ to be defined by $J_{\bigcirc, n+1}(q) = (-1)^n \langle n \rangle$, 
we obtain the expression
\[\langle n \rangle J'_{K, n}(q) = (-1)^nJ_{K, n+1}(q).\]
Furthermore, since the colored Jones function is multiplicative for connected sums, so is the normalized colored Jones function.  
That is, for knots $K_1$ and $K_2$ we have 
\[J'_{K_1 \sharp K_2, n}(q) = J'_{K_1, n}(q)J'_{K_2, n}(q).\]

\begin{lemma}
\label{condition sum}
If $K_1, K_2 \in \mathcal{K}$, 
then $K_1 \sharp K_2 \in \mathcal{K}$. 
\end{lemma}

\begin{proof}
[Proof of Lemma~\ref{condition sum}] 
Since the trivial knot is the identity for the connected sum operation, we may assume neither $K_1$ nor $K_2$ is trivial. 
By Theorem~\ref{strong_slope_conjecture_sum}, 
$K_1 \sharp K_2$ satisfies the Strong Slope Conjecture. 
So it remains to show that $K_1 \sharp K_2$ satisfies Condition $\delta$ and the  Sign Condition. 

Recall first that 
\[\langle n \rangle J'_{K, n}(q) = (-1)^nJ_{K, n+1}(q)\]
and that the normalized colored Jones function satisfies 
\[J'_{K_1 \sharp K_2, n}(q) = J'_{K_1, n}(q)J'_{K_2, n}(q).\]
Then we have
\begin{align*}
\langle n \rangle J'_{K_1 \sharp K_2, n}(q) 
& = \langle n \rangle J'_{K_1, n}(q)J'_{K_2, n}(q)\\ 
& = \langle n \rangle \frac{(-1)^n}{\langle n \rangle} J_{K_1, n+1}(q) \frac{(-1)^n}{\langle n \rangle} J_{K_2, n+1}(q)\\
& = \frac{1}{\langle n \rangle} J_{K_1, n+1}(q) J_{K_2, n+1}(q).
\end{align*} 
Since $\langle n \rangle J'_{K_1 \sharp K_2, n}(q) = (-1)^n J_{K_1 \sharp K_2, n+1}(q)$, 
we have  
\[\langle n-1 \rangle J_{K_1 \sharp K_2, n}(q) = (-1)^{n-1}J_{K_1, n}(q) J_{K_2, n}(q).\]
This implies 
\[\delta_{K_1 \sharp K_2}(n) = (a_1+a_2)n^2+(b_1+b_2-\frac 1 2)n+(c_1(n)+c_2(n)+\frac 1 2),\]
and 
\[\varepsilon_n(K_1 \sharp K_2) = \varepsilon_n(K_1)\varepsilon_n(K_2)
.\]

Let us see that $K_1 \sharp K_2$ satisfies Condition $\delta$. 
Since the  period of $c_i(n)$ is at most $2$, 
$c_1(n)+c_2(n)+\frac 1 2$ has period at most $2$, 
and hence $\delta_{K_1 \sharp K_2}(n)$ has also period $\le 2$. 
Since $b_1\le 0$ and $b_2 \le 0$, 
$(b_1 + b_2) - \frac 1 2 \le 0$. 
Since $4a_1$ and $4a_2$ are integers, so is $4(a_1+a_2)$. 

Finally we check the Sign Condition for $K_1 \sharp K_2$. 
Since $K_1$ and $K_2$ belong to $\mathcal{K}$, 
$\varepsilon_m(K_i) = \varepsilon_n(K_i)$ 
 for $m \equiv n\ \mathrm{mod}\ 2$, 
and hence 
$\varepsilon_m(K_1 \sharp K_2) = \varepsilon_n(K_1 \sharp K_2)$ 
 for $m \equiv n\ \mathrm{mod}\ 2$. 
\end{proof}

\begin{lemma}
\label{condition cable}
If $K \in \mathcal{K}$, 
then $K_{p, q} \in \mathcal{K}$. 
\end{lemma}

\begin{proof}[Proof of Lemma~\ref{condition cable}] 
If $K$ is trivial, then its cables $K_{p,q}$ are torus knots.
It is observed in \cite[p.924]{KT} that any nontrivial torus knot satisfies  Condition $\delta$. 
Furthermore, Case 2 in the proof of \cite[Theorem 3.9]{KT} shows that 
any nontrivial torus knot satisfies the Strong Slope Conjecture. 
As we observe in Proposition~\ref{ab-twisted-ab-torus} below, 
torus knots also satisfy the Sign Condition. 
Hence torus knots belong to $\mathcal{K}$.  
Thus we may assume $K$ is non-trivial. 

Then for non-trivial $K \in \mathcal{K}$, that $K_{p, q} \in \mathcal{K}$ follows from Proposition~\ref{cable-max}.
\end{proof} 

Here we show that a torus knot satisfies the Sign Condition 
by calculating the coefficient of the term of the maximum degree of its colored Jones polynomial. 

\begin{proposition}
\label{ab-twisted-ab-torus} 
Let $a$ and $b$ be coprime integers with $a>b>1$. 
Then the coefficient of the term of the maximum degree of the colored Jones polynomial ${J}_{K,n}(q)$ 
of $K=T_{a,b}$ is $1$ if $n$ is odd and $-1$ if $n$ is even. 
The coefficient of the term of the maximum degree of the colored Jones polynomial ${J}_{K,n}(q)$ 
of $K=T_{-a,b}$ is 1.
\end{proposition} 

\begin{proof}
In the following we observe that the coefficient of the term of the maximum degree of the normalized colored Jones polynomial ${J'}_{K,n}(q)$ 
of $K=T_{a,b}$ is $1$ if $n$ is even and $-1$ if $n$ is odd, and 
the coefficient of the term of the maximum degree of the normalized colored Jones polynomial ${J'}_{K,n}(q)$ 
of $K=T_{-a,b}$ is $1$. 
Then the result follows from the formula
\[J_{K, n}(q) = (-1)^{n-1}\langle n-1 \rangle J'_{K, n-1}(q)=(q^{\frac {n-1} 2}+ \cdots +q^{-\frac {n-1} 2})J'_{K, n-1}(q).\]

First we note that for any knot $K$ with mirror $K^*$, 
the maximum degree of ${J'}_{K^*,n}(q)$ is the minimum degree of ${J'}_{K,n}(q^{-1})$.   
So instead of determining the coefficient of the term of the maximum degree of the colored Jones polynomial ${J'}_{K,n}(q)$ of $K=T_{-a,b}$, 
we instead determine  the coefficient of the term of the minimum degree of the colored Jones polynomial ${J'}_{K,n}(q)$ of $K=T_{a,b}$.

The normalized colored Jones polynomial of $K = T_{a, b}$ is explicitly computed in \cite{Mor}: 
\begin{equation}
\label{J'}
J'_{K,n}(q) = \frac {q^{\frac 1 4 abn(n+2)}}{q^{\frac {n+1}2}-q^{-\frac {n+1}2}}
  \sum_{k=-\frac n 2}^{\frac n 2} 
    (q^{-abk^2+(a-b)k+\frac 1 2}-q^{-abk^2+(a+b)k-\frac 1 2}). 
\end{equation}

First we consider the case where $n$ is even.
Then $k$ is an integer in the summand. 
We define the functions $f_\pm (\ell)$ on $\Z$ by 
\[
 f_{\pm}(\ell):=-ab \ell^2+(a\mp b)\ell \pm \frac 1 2. 
\]
Since 
\[
f_{\pm}(\ell)=-ab(\ell- \frac {a\mp b}{2ab})^2+\frac {(a\mp b)^2}{4ab}\pm \frac 1 2
\]
and $0<\frac {a\mp b}{2ab}<\frac 1 2$, $f_\pm (\ell)$ is maximized at $\ell=0$ and 
$f_- (0)<f_+ (0)=\frac 1 2$. 
Hence the maximum degree of $J'_{K,n}(q)$ for even $n$ is calculated  by 
\[
\frac 1 4 abn(n+2)-\frac {n+1} 2+\frac 1 2= \frac {ab}{4} n^2+\frac {ab-1}{2} n. 
\]

Since $J'_{K, n}(q)$ is a Laurent polynomial, 
we may write 
\[J'_{K, n}(q) = A q^{\frac {ab}{4} n^2+\frac {ab-1}{2} n} + [\textrm{lower degree terms}]\] 
for some integer $A$.  
Then following $(\ref{J'})$, 
we have 
\[
(q^{\frac {n+1}2}-q^{-\frac {n+1}2}) (A q^{\frac {ab}{4} n^2+\frac {ab-1}{2} n} + [\textrm{lower degree terms}]) 
= (q^{\frac 1 4 abn(n+2)}) q^{\frac{1}{2}} + [\textrm{lower degree terms}].
\]
This shows $A=1$ and the term of the maximum degree is $q^{\frac {ab}{4} n^2+\frac {ab-1}{2} n}$.

Moreover, $f_\pm (\ell)$ is minimized at $\ell=-\frac n 2$ and 
$f_+ (-\frac n 2)>f_- (-\frac n 2)=-\frac {ab} 4 n^2-\frac {a+b} 2n -\frac 1 2$. 
Hence the minimum degree of $J'_{K,n}(q)$ for even $n$ is calculated  by 
\[
\frac 1 4 abn(n+2)+\frac {n+1} 2-\frac {ab} 4 n^2-\frac {a+b} 2 n -\frac 1 2
= \frac {ab-a-b+1} 2 n=\frac {(a-1)(b-1)} 2 n, 
\]
and the term of the minimum degree is $q^{\frac {(a-1)(b-1)} 2 n}$.

Now we consider the case where $n$ is odd.
Then $k$ is a half-integer in the summand. 
We define the functions $f_\pm (\ell)$ on $\Z + \frac 1 2$ by 
\[
 g_{\pm}(\ell):=-ab \ell^2+(a\mp b)\ell \pm \frac 1 2. 
\]
Since 
\[
g_{\pm}(\ell)=-ab(\ell- \frac {a\mp b}{2ab})^2+\frac {(a\mp b)^2}{4ab}\pm \frac 1 2
\]
and $0<\frac {a\mp b}{2ab}<\frac 1 2$, $g_\pm (\ell)$ is maximized at $\ell=\frac 1 2$ and 
$g_+ (\frac 1 2)<g_- (\frac 1 2)=-\frac {ab} 4+\frac {a+ b} 2 - \frac 1 2$. 
Hence the maximum degree of $J'_{K,n}(q)$ for even $n$ is calculated  by 
\[
\frac 1 4 abn(n+2)-\frac {n+1} 2-\frac {ab} 4+\frac {a+ b} 2 - \frac 1 2
= \frac {ab} 4 n^2+\frac {ab-1} 2 n-\frac {(a-2)(b-2)} 4\, 
\]
and the term of the maximum degree is $-q^{\frac {ab} 4 n^2+\frac {ab-1} 2 n-\frac {(a-2)(b-2)} 4}$. 

Moreover, $g_\pm (\ell)$ is minimized at $\ell=-\frac n 2$ and 
$g_+ (-\frac n 2) > g_- (-\frac n 2)=-\frac {ab} 4 n^2-\frac {a+b} 2n -\frac 1 2$. 
Hence the minimum degree of $J'_{K,n}(q)$ for odd $n$ is calculated  by 
\[
\frac 1 4 abn(n+2)+\frac {n+1} 2-\frac {ab} 4 n^2-\frac {a+b} 2 n -\frac 1 2
= \frac {ab-a-b+1} 2 n=\frac {(a-1)(b-1)} 2 n, 
\]
and the term of the minimum degree is $q^{\frac {(a-1)(b-1)} 2 n}$. 
\end{proof}

\subsection{$B$--adequate knots satisfy the Sign Condition}

Let us turn to see that $B$--adequate knots, 
and hence adequate knots, also belong to $\mathcal{K}$. 
It is known by \cite[Theorem~3.9]{KT} that $B$--adequate knots satisfy the Strong Slope Conjecture. 
For a $B$--adequate diagram $D$ with $c_+$ positive crossings of a $B$--adequate knot $K$, 
\cite[Lemma~3.6]{KT} shows that $\delta_K(n)=\frac {c_+}{2}n^2+bn^2+c$, so $4a \in \Z$.
Moreover, from \cite[Lemma~3.8]{KT},  we have that $b \le 0$.
Thus $B$--adequate knots satisfy Condition $\delta$. 
It remains to observe that $B$--adequate knots also 
satisfy the Sign Condition. 

We begin by reviewing a presentation of $J_{K,n}(q)$ in terms of Chebychev polynomials $S_n(x)$ for $n \ge 0$ \cite{Lic}. 
The polynomial $S_n(x)$ is defined recursively as follows:
\begin{equation}
\label{cheb} 
S_{n+1}(x)=x S_n(x)-S_{n-1}(x), \qquad S_1(x)=x, \qquad S_0(x)=1
\end{equation}
Let $D$ be a diagram of a knot $K$. 
For an integer $m >0$, 
let $D^m$ denote the diagram obtained from $D$ by taking $m$ parallel copies of $K$. 
This is the $m$--cable of $D$ using the blackboard framing; if $m = 1$, then $D^1=D$. 
The Kauffman bracket is the function 
$\langle \cdot \rangle \colon
\{ \mathrm{unoriented \ link \ diagrams }\} \to \Z [t^{\pm 1}]
$
satisfying 
\begin{eqnarray*}
 && (1)\, \langle \begin{minipage}[1cm]{1cm}
      \begin{center}      
    \includegraphics[scale=0.2]{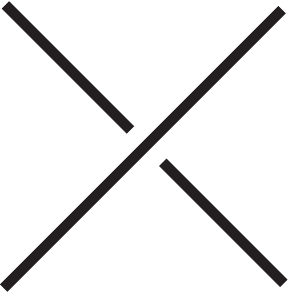}
     \end{center}    
     \end{minipage}
       \rangle
    =t \,  \langle
    \begin{minipage}[1cm]{1cm}
    \begin{center}
    \includegraphics[scale=0.2]{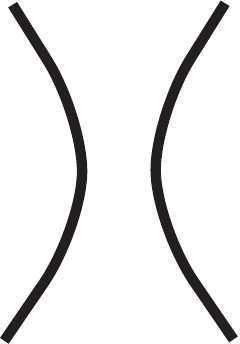}
    \end{center}
    \end{minipage}
     \rangle \, 
    +t^{-1} \,  \langle
    \begin{minipage}[1cm]{1cm}
    \begin{center}
    \includegraphics[scale=0.2]{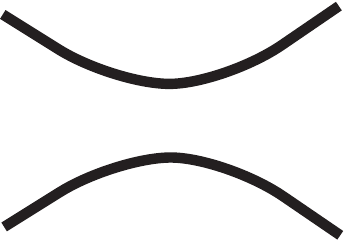}
    \end{center}
    \end{minipage}
     \rangle\\
 && (2) \, \langle D  \sqcup \bigcirc\rangle \ = (-t^2-t^{-2}) \ \langle D \rangle. 
\end{eqnarray*}
It is normalized so that the bracket of the empty link is $1$. 
Let $w(D)$ be the writhe of $D$. 
Then  the colored Jones polynomial of $K$ is given by 
\begin{equation}
\label{jones-skein}
J_{K, n}(q)=(-1)^{n-1}((-1)^{n-1}q^{(n^2-1)/4})^{w(D)}\langle S_{n-1}(D)\rangle|_{t = q^{-1/4}},
\end{equation} 
where $S_{n}(D)$ is a linear combination of blackboard cablings of $D$, obtained via the equation 
(\ref{cheb}), and the notation $\langle S_{n}(D)\rangle$ means to extend the Kauffman bracket  linearly.

\begin{proposition}\label{prop-BadequateSignCond}
Let $K$ be a $B$--adequate knot. 
Then $\varepsilon_m(K)=\varepsilon_n(K)$ if $m \equiv n \pmod 2$, 
namely a $B$--adequate knot satisfies the Sign Condition. 
\end{proposition}

\begin{proof} 
Let $X$ be the set of crossings of a diagram $D$. 
Let $c(D)$ be the number of crossings of $D$. 
A state for $D$  is a function $s \colon X \to \{\pm 1\}$. 
For each $\pm = +$ or $-$, 
 we denote by $s_{\pm}$ the special state $s$ with $s(x)=\pm 1$ for every crossing $x$. 
For a state $s$, let $sD$ be the diagram constructed from $D$ by doing 
$s(x)$--smoothing (see e.g.\ \cite{Lic}) at every crossing $x$. 
Then $sD$ consists of disjoint simple closed curves on $S^2$. 
Let $v_s(D)$ be the number of connected components of $sD$ and 
$\sigma(s):=\sum_{x \in X} s(x)$. 
We have that $\sigma(s_+)=c(D)$ and $\sigma(s_-) = -c(D)$. 

The Kauffman bracket $\langle D \rangle$ for a diagram $D$ is given by
\begin{equation}\label{Kauff}
\langle D \rangle = \sum_{s : \mathrm{state}} \langle sD \rangle 
= \sum_{s : \mathrm{state}} t^{\sigma(s)}(-t^{2}-t^{-2})^{v_s(D)}.
\end{equation} 

Recall the following facts; see \cite[Lemma~2.17, the proof of Proposition~4.6]{Le-lec}. 
The fact (1) follows from the equation (\ref{cheb}). 
\begin{enumerate}
\item $S_n(D)=D^n+\mathrm{lower \ degree \ cablings \ of}\ D$.  
\item If $D$ is a $B$--adequate diagram, 
then $D^n$ is also a $B$--adequate diagram. 
\item If $D$ is a $B$--adequate diagram, 
then $c(D^n)=n^2c(D)$ and $v_{s_{-}}(D^n)=nv_{s_{-}}(D)$. 
\item If $D$ is a $B$--adequate diagram, 
then $\mathrm{deg}_{-}[\langle S_{n}(D) \rangle] 
= \mathrm{deg}_{-}[\langle D^n \rangle]$, 
where $\mathrm{deg}_{-}$ means  minimum $t$--degree. 
\end{enumerate} 

Moreover, for a $B$--adequate diagram $D$, 
the minimum $t$--degree of $\langle D \rangle$ coincides with that of $\langle s_{-}D \rangle$; 
see the paragraph before \cite[Theorem~2.16]{Le-lec}. 
This then implies that the maximum $q$--degree of $\langle D \rangle|_{t = q^{-1/4}}$ coincides with that of $\langle s_{-}D \rangle|_{t = q^{-1/4}}$.  
Hence we have the formula for the maximum $q$--degree: 
\[d_+[\langle D \rangle|_{t = q^{-1/4}}] = d_+[\langle s_{-}D \rangle|_{t = q^{-1/4}}].\] 
Combining this with formula (\ref{Kauff}), 
we have the following for $\langle D \rangle|_{t = q^{-1/4}}$: 

$\bullet$ 
The term of maximum $q$--degree of $\langle D \rangle|_{t = q^{-1/4}}$
is $q^{-\frac{\sigma(s_{-})}{4}}(-q^{\frac{1}{2}})^{v_{s_{-}}(D)}
= q^{ \frac{c(D)}{4}}(-q^{\frac{1}{2}})^{v_{s_{-}}(D)}$.

\noindent
In particular, apply this to the diagram $D^n$ to obtain for $\langle D^n \rangle$:  
\smallskip

$\bullet$ 
The term of maximum $q$--degree of $\langle D^n \rangle|_{t = q^{-1/4}}$ 
is $q^{ \frac{c(D^n)}{4}}(-q^{\frac{1}{2}})^{v_{s_{-}}(D^n)}
= q^{ \frac{n^2c(D)}{4}}(-q^{\frac{1}{2}})^{nv_{s_{-}}(D)}$.

\noindent
Therefore, if $D$ is a $B$--adequate diagram, 
following (1) and (4) we have: 

$\bullet$ The term of the maximum $q$--degree of $\langle S_n(D) \rangle|_{t = q^{-1/4}}$  
coincides with the term of the maximum $q$--degree of $\langle D^n \rangle|_{t = q^{-1/4}}$. 

\noindent
Hence 
the term of the maximum $q$--degree of $\langle S_n(D) \rangle|_{t = q^{-1/4}}$ 
is 
$q^{\frac{n^2c(D)}{4}}(-q^{\frac{1}{2}})^{nv_{s_{-}}(D)}$.

This, together with (\ref{jones-skein}), 
shows that if $K$ is $B$-adequate, 
then the sign of the coefficient of the term of the maximum $q$--degree of $J_{K, n}(q)$ is 
\[(-1)^{n-1}(-1)^{(n-1)w(D)}(-1)^{(n-1)v_{s_{-}}(D)} = (-1)^{(n-1)(w(D)+v_{s_{-}}(D)+1)}.\] 
Hence $\varepsilon_m(K) = \varepsilon_n(K)$ 
if $m \equiv n \pmod 2$.  
\end{proof}

\medskip

\subsection{Some knots that do not satisfy the Sign Condition.}
\label{no_Sign_Condition}

A computer experiments 
suggest the knots $8_{20}$, $9_{43}$, and $9_{44}$ do not satisfy the Sign Condition. 
The following table gives the sign $\varepsilon_n(K)$
for $K=8_{20}$, $9_{43}$, $9_{44}$ and $1\le n \le 6$.
\begin{table}[htb]
\begin{center}
  \begin{tabular}{@{}lcccccc@{}} 
  \toprule
	$K$ & $\varepsilon_1(K)$ & $\varepsilon_2(K)$ & $\varepsilon_3(K)$ & $\varepsilon_4(K)$ & $\varepsilon_5(K)$ & $\varepsilon_6(K)$  \\
   \midrule
    $8_{20}$ & $-$ & $-$ & $+$ & $-$ & $-$ & $+$ \\ 
    $9_{43}$ & $-$ & $+$ & $+$ & $-$ & $+$ & $+$ \\ 
    $9_{44}$ & $+$ & $-$ & $-$ & $-$ & $+$ & $+$ \\ 
    \bottomrule
  \end{tabular}

\end{center}
\end{table}

We computed the colored Jones polynomials for these knots
using Mathematica package \verb!KnotTheory`! and its program \verb!ColouredJones! \cite{BarKT,Mat} in order to determine the signs $\epsilon_n(K)$.

In the above examples all knots have $\delta_K(n)$ with period $3$. 
It may be reasonable to ask: 

\begin{question}
\label{sign_condition_period2}
Let $K$ be a knot such that 
$\delta_K(n)$ has period $\le 2$. 
Then does $K$ satisfy the Sign Condition?
\end{question}

\bigskip

\appendix

\section{}\label{app:monosloped}

While \cite[Proposition~3.2]{KT} addresses a cabling formula for knots $K$ that have $\delta_K(n) = a(n)n^2 + b(n)n + c(n)$ with period $\leq 2$,  \cite[Proposition~4.4]{KT} gives a variant that allows for an arbitrary period at the expense of requiring $a(n)$ to be a constant function.
Similar to our revision of \cite[Proposition~3.2]{KT} with Proposition~\ref{max-degree} in Section~\ref{cabling}, 
in this appendix we correct and extend the statement of \cite[Proposition~4.4]{KT} with our Proposition~\ref{max-degree-monoslope} and give a revision of its proof in the spirit of the original. 
This then permits us to update \cite[Theorem~4.1]{KT} in Proposition~\ref{cable-max-mono}.

\begin{proposition}
\label{max-degree-monoslope} 
Let $K$ be a knot such that 
$\delta_K(n)=a n^2+b(n)n+c(n)$ has period $\pi$ with $b(n)\le 0$.
Suppose $\frac{p}{q} \neq 4 a$ if $b(n)=0$, 
and $K$ satisfies the Sign Condition. 
Let $M_1:=\max \{ |b(i)-b(j)| \, : \, i \equiv j \pmod 2\}$. 

Then 
$
\delta_{K_{p,q}}(n) =A n^2+B(n)n+C(n)
$
has period $\le \pi$ with
\[
 A   \in \{q^2 a, \frac {pq} 4\}
\quad and \quad 
 B(n)\le 0.
\]
Explicitly, we have 
\begin{equation*}
  \delta_{K_{p,q}}(n)=
   \left\{ \begin{array}{ll} 
           q^2 a n^2 +\left(q b(i)+\frac {(q-1)(p-4qa)} 2\right) n & \\
          \quad\quad +\left(a\cdot(q-1)^2-(b(i)+\frac p 2)(q-1)+c(i)\right) 
          	& \mbox{for } \frac p q < 4a-M_1\\
           \frac {pq(n^2-1)} 4 + C_{{\sigma}}(K_{p,q})  
           	&  \mbox{for }\frac p q  \geq  4a,
          \end{array} \right. 
\end{equation*}
where $i \equiv_{(\pi)} q(n-1)+1$, $\sigma \equiv_{(2)} n$, 
and $C_{{\sigma}}(K_{p,q})$
is a number that only depends on the knot $K$, 
the numbers $p$ and $q$, and the parity ${\sigma}$ of $n$. 
Furthermore, 
$K_{p,q}$ satisfies the Sign Condition. 
\end{proposition}

\begin{proof}
For notational concision, 
let us also write the periodic coefficients of $\delta_K(m)$ as $b_m = b(m)$ and $c_m=c(m)$ for integers $m$ considered mod $\pi$. 
(Note that $a(m) = a$ is constant by assumption.)
As in the proof of Proposition~\ref{max-degree} 
for each integer $n>0$, 
 let $\cS$ be the finite set of all numbers $k$ such that
\[
|k|\le \frac {n-1} 2 \quad \mathrm{and} \quad 
k\in \left\{ \begin{array}{ll} 
            \Z & \mathrm{if} \, n \,\mathrm{is \, odd},\\
            \Z+\frac 1 2 & \mathrm{if} \, n \,\mathrm{is \, even}.
            \end{array}
\right.
\]
That is,
\[ \cS = \left\{-\frac{n-1}{2},\ -\frac{n-1}{2}+1,\ -\frac{n-1}{2}+2,\ \dots,\ \frac{n-1}{2} -1,\ \frac{n-1}{2}\right\}.\]
Then for $n>0$ we have the following; see (\ref{cabling.formula}). 
\[
J_{K_{p,q}, n}(v)=v^{pq(n^2-1)/4}\sum_{k\in \cS} v^{-pk(qk+1)}J_{K, 2qk+1}(v), 
\]
where we use the convention that $J_{K, -m}(v) = -J_{K, m}(v)$ for integers $m>0$. 

Recall that,
for integers $n>0$ of a given parity, 
the parity of $2qk+1$ for $k \in \cS$ is constant. 
More precisely, if $n$ is odd, then $2qk + 1$ is odd, 
and if $n$ is even, then $2qk+1$ is odd or even according to whether $q$ is even or odd, respectively.
As in the proof of Proposition~\ref{max-degree}, due to the Sign Condition we still have equation~\ref{cabling.max}:
\[d_+[J_{K_{p,q}, n}(v)] = pq(n^2-1)/4 + \max_{k \in {\cS}} \{-pk(qk+1) + d_+[J_{K, |2qk+1|}(v)] \}\]

\medskip
First define 
\[
f(k) = -pk(qk+1) + d_+[J_{K, |2qk+1|}(v)] 
\]
for $k\in \cS$.   
Set $N_K \geq 0$ to be the first integer such that $d_+[J_{K, |m|}(v)] =\delta_K(|m|)$ for all integers $m$ with $|m| \ge 2q N_K +1$. 
Noting that $|2q(-N_K - \frac{1}{2}) + 1| \geq 2qN_K+1 > |2q(-N_K) + 1|$, 
partition $\cS$ into the three subsets 
\[\cS^- = \cS \cap (-\infty, -N_K-\frac{1}{2}], \quad \cS^0 = \cS \cap (-N_K-\frac{1}{2}, N_K), \quad \mbox{ and } \cS^+ = \cS \cap [N_K, \infty). \]
Note that when $n=1$, $\cS = \{ 0 \}$ and $\cS^- = \emptyset$.

Then, considering the quadratic quasi-polynomials for integers and half-integers $k$
\begin{align*}
h^+(k) &= -pk(qk+1) + \delta_K(2qk+1) \\
		&= (-pq + 4 q^2 a) k^2 + (-p+4q a + 2q b_{m}) k +(a + b_{m}+c_{m})\\
		& \quad \mbox{for } k \geq 0 \mbox{ and } m \equiv_{(\pi)} 2qk+1
\end{align*}

and
\begin{align*}
h^-(k)  &= -pk(qk+1) + \delta_K(|2qk+1|)\\
         & = -pk(qk+1) + \delta_K(-2qk-1)\\
	&= (-pq + 4 q^2 a) k^2 + (-p+4q a - 2q b_{m}) k +(a - b_{m}+c_{m})\\
	& \quad \mbox{for } k < 0 \mbox{ and } m \equiv_{(\pi)} |2qk+1|,
\end{align*}
define the quadratic real polynomials $h_m^{\pm}(x)$ for integers $m\pmod \pi$ by
\[h_m^{\pm}(x)=(-p q+4 q^2 a)x^2+(-p+4q a \pm 2q b_m)x+a\pm b_m+c_m.\]
Hence for integers and half-integers $k$, 
we have $h^\pm(k) = h^\pm_{m}(k)$ where $m \equiv_{(\pi)} |2qk+1| $ and $\pm$ means $+$ if $k \geq 0$ and $-$ if $k <0$.
Thus, on the subsets $\cS^\pm$, we have 
\[
f(k)= h_m^\pm(k)  \quad \mathrm{ if }\, k\in \cS^\pm\ \mathrm{ and }\ m\equiv_{(\pi)} |2qk+1|.
\]
While we have little information about $f(k)$ for $k \in \cS^0$, 
it belongs to only a finite set of values since 
\[\cS^0 \subset \left[-N_K, N_K - \frac 1 2\right] \cap \frac 1 2 \Z\]
for all $n>0$.

\bigskip

Recall that $f(k)$ is defined on half-integers $k$ if $n$ is even, 
and defined on integers $k$ if $n$ is odd. 
So to clarify the role of the parity of $n$, 
for each parity $\sigma \in \{0,1\}$ we define the function $f_{{\sigma}}(k)$ so that 
$f(k) = f_{{\sigma}}(k)$ if ${\sigma} \equiv_{(2)} n$. 
Then $f_0$ is defined on half-integers, while $f_1$ is defined on integers. 
Define $\cS^\pm(i) = \{ k \in \cS^\pm \colon  |2qk+1| \equiv_{(\pi)} i \} \subset \cS^\pm$. 
More explicitly, this is the function 
\[
f_{{\sigma}}(k)=\begin{cases}
h_i^+(k) &  \mathrm{if} \, k\in \cS^+(i),\\
h_i^-(k) &  \mathrm{if} \, k\in \cS^-(i), 
\end{cases}
\]
where $i \equiv_{(\pi)} |2qk+1|$.

Henceforth regard ${\sigma}$ as fixed choice of parity.
Using (\ref{cabling.max}) and $f_{{\sigma}}(k)$, 
we now proceed to determine $d_+[J_{K_{p,q}}(n)]$ for suitably large $n$ such that $n \equiv_{(2)}{\sigma}$. 

\smallskip

\noindent 
{\bf Case 1.} 
 Assume $p-(4a-M_1)q<0$. 
Then $-pq+4q^2a>0$,
and so the functions given by the quadratic polynomials $h_i^+(x)$ and $h_i^-(x)$ are concave up.
Hence, for any sufficiently large integer $n$,  
$h_i^+(k)$ is maximized on $\cS^+(i)$ at $k_i^+ = \max \cS^+(i) = \frac {n-1} 2 - J^+_i$ 
and $h_i^-(k)$ is maximized on $\cS^-$ at $k_i^- = \min \cS^-(i) = -\frac{n-1} 2 + J^-_i$ for some integers $0 \leq J^\pm_i$.  
We note that $\max \cS^+ = \max \cS^+(i)=\frac{n-1}{2}$ for $i$ with $i \equiv_{(\pi)} q(n-1)+1$.

Note that 
\begin{align*}
h_i^+(k_i^+)-h_j^+(k_j^+) &= 
((-p q+4 q^2 a)(k_i^+)^2+(-p+4q a +  2q b_i)k_i^+ + a + b_i+c_i)\\
& \phantom{offset space}
-
((-p q+4 q^2 a)(k_j^+)^2+(-p+4q a +  2q b_j)k_j^+ + a + b_j+c_j)\\
&=(-p q+4 q^2 a)((k_i^+)^2 - (k_j^+)^2) - (-p+4q a)(k_i^+ - k_j^+) + 2q( b_i k_i^+ -b_j  k_j^+) \\
& \phantom{offset space}+ b_i-b_j+c_i-c_j. 
\end{align*}
Especially, 
\begin{align*}
h_i^+(\frac {n-1}2)-h_j^+(\frac {n-1}2- 1) 
&=(-p q+4 q^2 a)(n-2) - (-p+4q a) + q( b_i-b_j  ) (n-1)+2qb_j\\
& \phantom{offset space}+ b_i-b_j+c_i-c_j\\
&= q(-p +4qa + (b_i-b_j))n+\alpha\\
& > q(-p +4qa -M_1q)n+\alpha \to \infty \quad (n \to \infty)
\end{align*}
for a constant $\alpha$. 
Therefore, for any sufficiently large $n$, 
$h_i^+(\frac {n-1}2) > h_j^+(\frac {n-1}2 - 1) \ge  h_j^+(\frac {n-1}2 - J^+_j)$.

Note that 
\begin{align*}
h_i^+(k_i^+)-h_j^-(k_j^-) &= 
((-p q+4 q^2 a)(k_i^+)^2+(-p+4q a +  2q b_i)k_i^+ + a + b_i+c_i)\\
& \phantom{offset space}
-
((-p q+4 q^2 a)(k_j^-)^2+(-p+4q a -  2q b_j)k_j^- + a - b_j+c_j)\\
&=(-p q+4 q^2 a)((k_i^+)^2 - (k_j^-)^2) - (-p+4q a)(k_i^+ - k_j^-) + 2q( b_i k_i^+ +b_j  k_j^-) \\
& \phantom{offset space}+ b_i+b_j+c_i-c_j. 
\end{align*}
Especially, 
\begin{align*}
h_i^+(\frac {n-1}2)-h_j^-(-\frac {n-1}2) 
&=(-p+4q a)(n-1) + q( b_i-b_j  ) (n-1)+ b_i+b_j+c_i-c_j\\
&=-(p-4qa+q(b_j-b_i))n+\beta\\
&>-(p-4qa+M_1 q)n+\beta \to \infty \quad (n \to \infty)
\end{align*}
for  a constant $\beta$. 
Hence, for any sufficiently large $n$, 
$h_i^+(\frac {n-1}2) > h_j^-(-\frac {n-1}2) \ge h_j^-(-\frac {n-1}2 + J^-_j)$.
Therefore, $f_{{\sigma}}(k)$ is maximized on the set $\cS^+ \cup \cS^-$ 
at $k=\frac {n-1} 2$. 

Since the elements of $\cS^0$ belong to a fixed finite set that is independent of $n$, 
the maximum of $f_{{\sigma}}(k)$ on $\cS^0$ has an upper bound that is independent of $n$.
Thus, for a sufficiently large integer $n$, 
we can be assured that $h_i^+(\frac {n-1} 2)$ exceeds this bound.  
Hence
\[ \max_{k\in \cS} f_{{\sigma}}(k) = f_{{\sigma}}(\frac{n-1} 2) = h_i^+(\frac {n-1} 2). \]
Then, Formula~(\ref{cabling.formula}) implies that for sufficiently large integer $n$
\begin{eqnarray*}
	d_+[J_{K_{p,q}}(n)]&=& \frac {pq(n^2-1)} 4+ h_i^+(\frac {n-1} 2) \\
	&=& q^2 a n^2+\left(qb_i+\frac {(q-1)(p-4qa)} 2\right)n \\
	&& \quad \quad +\left(a(q-1)^2-(b_i+\frac p 2)(q-1)+c_i\right). 
\end{eqnarray*}
for $i \equiv_{(\pi)} q(n-1)+1$.

\smallskip

\noindent 
{\bf Case 2.}   
Assume $p/q>4a$.  
Then $-pq+4q^2a<0$, 
and so the function given by the quadratic polynomial $h_i^+(x)$ is concave down 
and attains its maximum at 
\[
x=x_0:=-\left( \frac 1 {2q} + \frac {b_i}{-p+4qa}\right). 
\]
Since $b_i\le 0$, we have $x_0<0$. 
This implies that $h_i^+(x)$ is a strictly decreasing function on $[0,\infty )$.   
Similarly, the quadratic polynomial $h_i^-(x)$ is concave down 
 and attains its maximum at 
\[ 
x=x_0':=-\left( \frac 1 {2q} - \frac {b_i}{-p+4qa}\right). 
\]
Since $b_i\le 0$, we have $x_0' > -\frac 1 2$. 
This implies that $h_i^-(x)$ is a strictly increasing function on $(-\infty,-\frac 1 2]$. 
Thus $h_i^+(k)$ is maximized on $\cS^+(i)$ at $k_i^+ =\min \cS^+(i)$
 and $h_i^-(k)$ is maximized on $\cS^-(i)$ at $k_i^- = \max \cS^-(i)$.

Since $| \cS^0 | \le 2N_K$, 
there are at most $2N_K$ values $f(k)$ for $k \in \cS^0$, 
and thus we may take $M_0 = \max \{ f(k)\ |\ k \in \cS^0 \}$. 
Now let us put 
$\displaystyle C_{{\sigma}}(K_{p, q}) = \max_{0\le i <\pi}\{ f_{{\sigma}}(k_i^+), f_{{\sigma}}(k_i^-), M_0 \}$. 
Formula~(\ref{cabling.max}) implies that 
\[
d_+[J_{K_{p,q}, n}(v)]=\frac {pq(n^2-1)} 4+ C_{{\sigma}}(K_{p,q})
\]
for sufficiently large integer $n$ 
with ${\sigma} \equiv_{(2)} n$. 

Note that $B(n)=0$.  Hence the conclusion follows in this case too.

\smallskip

\noindent 
{\bf Case 3.} Assume $p/q=4a$ and $b_i<0$.
Then $-p+4qa=0$ so that
\[h_i^{\pm}(x)=\pm (2q b_i)x+a\pm b_i+c_i.\]
Since $q>1$ and $b_i<0$, $h_i^+(x)$ is strictly decreasing and $h_i^-(x)$ is strictly increasing.  
Thus $h_i^+(k)$ is maximized  on $\cS^+(i)$ at $k_i^+ =\min \cS^+(i)$ 
and $h_i^-(k)$ is maximized on $\cS^-(i)$ at $k_i^- = \max \cS^-(i)$.

As in Case 2, 
let $M_0 = \max \{ f_{{\sigma}}(k)\ |\ k \in \cS^0 \}$ and put $\displaystyle C_{{\sigma}}(K_{p, q}) = \max_{0\le i <\pi}\{ f_{{\sigma}}(k_i^+), f_{{\sigma}}(k_i^-), M_0 \}$. 
Then 
\[
d_+[J_{K_{p,q}, n}(v)]=\frac {pq(n^2-1)} 4+ C_{{\sigma}}(K_{p,q})
\]
for sufficiently large integer $n$ 
with ${\sigma} \equiv_{(2)} n$.  

\medskip

It remains to show that $\varepsilon_m(K_{p, q}) = \varepsilon_n(K_{p, q})$ for $m \equiv n\ \mathrm{mod}\ 2$. 
The identical argument in the proof of Proposition~\ref{max-degree} shows 
that the colored Jones polynomial of $J_{K_{p,q}, n}(v)$ has the required property of $\varepsilon_n(K_{p,q})$. 
\end{proof}

\begin{remark}
Let $K$ be the knot $8_{20}$.  
Then $d_+[J_{K, n}(v)] = \delta_K(n)$ for all $n > 0$ and it has period $3$ with constant $a(n) = a$ \cite{Garoufalidis}. 
So we may apply \cite[Proposition 4.4]{KT} to $K$. 
However, as shown in Example~\ref{no_Sign_Condition},  
$K$ does not satisfy the Sign Condition, 
and so we cannot apply Proposition~\ref{max-degree-monoslope} to $K$. 
\end{remark}

\begin{proposition}
\label{cable-max-mono}
Let $K$ be a knot such that
\begin{itemize}
\item $\delta_K(n)=a n^2+b(n)n+c(n)$ has period $\pi$, $b(n)\le 0$, and $4a \in \Z$,
\item $K$ satisfies the Sign Condition, and
\item the Jones slope $4a$ satisfies $SS(1)$ so that $K$ satisfies the Strong Slope Conjecture.
\end{itemize}
Let $M_1:=\max \{ |b(i)-b(j)| \, : \, i \equiv j \pmod 2\}$. \\
  Then for a non-trivial cable $K_{p,q}$ with $\frac p q \not \in [4a-M_1, 4a]$,
\begin{itemize}
\item $\delta_{K_{p,q}}(n)=An^2+B(n)n+C(n)$ has period $\le \pi$, $B(n) \le 0$, and $4A \in \Z$,
\item $K_{p,q}$ satisfies the Sign Condition, and  
\item the Jones slope $4A$ satisfies $SS(1)$ so that $K_{p, q}$ satisfies the Strong Slope Conjecture.
\end{itemize}
\end{proposition}

\begin{proof}
By hypothesis $\delta_K(n) = an^2 + b(n)n + c(n)$ has period $\pi$ with $4a \in \Z$ and $b(n)\leq 0$.  

Then, because $K$ also satisfies the Sign Condition, Proposition~\ref{max-degree-monoslope} shows that 
$\delta_{K_{p,q}}(n)=An^2+B(n)n+C(n)$ has period $\le \pi$ and $K_{p,q}$ satisfies the Sign Condition.

Case (1):  If $\frac p q < 4a-M_1$,
then $A=q^2a$ and $B(n)=qb(i)+\frac {(q-1)(p-4qa)}{2} $ 
with $i \equiv_{(\pi)} q(n-1)+1$, and so 
$4A\in \Z$ and $B(n) \le 0$. 
Since $q(n-1)+1 \equiv_{(\pi)} 1$ for $n \equiv_{(\pi)} 1$, 
we have 
\begin{equation*}
A=q^2 a, \, B(1)=qb(1)+\frac {(q-1)(p-4qa)} 2.
\end{equation*}
Since $4a$ satisfies $SS(1)$ for $K$, 
there is an essential surface $S_K$ in $E(K)$ with boundary slope $4a$ 
such that  $\displaystyle{\frac {\chi (S_K)}{|\bdry S_K |}=2b(1)}$. 
As in Case 1 of Proposition~\ref{max-degree}, an essential surface $S$ in $E(K_{p,q})$ realizing the boundary slope $4A=4qa^2$ is obtained from $S_K$ in the proof of \cite[Theorem 2.2]{KT}. 
Furthermore \cite[Corollary 2.8]{KT} shows 
 $|\bdry S|=|\bdry S_K|$ and 
\[
  \chi (S)=q \chi (S_K)+|\bdry S_K|(q-1)(p-4aq),
\]
so that we obtain:
\[
 \frac {\chi (S)}{|\bdry S|}=2b(1)q+(q-1)(p-4aq)=2B(1).
\]
Hence the Jones slope $4A$ satisfies $SS(1)$ for $K_{p,q}$. 

(2) if $\frac p q > 4a$, then $A=\frac {pq} 4$ and $B(n)=0$,  
 $4A \in \Z$ and $B(n) \le 0$. 
Then the cabling annulus $S$ is an essential surface with boundary slope $pq=4A$ and 
$\frac {\chi (S)}{|\bdry S|}=0=2B(1)$. 
Hence the Jones slope $4A$ satisfies $SS(1)$ for $K_{p,q}$. 
\end{proof}

\end{document}